 \documentclass{amsart}

\usepackage[latin1]{inputenc}
\usepackage{amssymb}
\usepackage{amsmath}
\usepackage{esint}
\usepackage{hyperref}
\usepackage{multirow}
\usepackage{multicol}
\usepackage{graphicx}
\usepackage{color}
\usepackage{url}

\hypersetup{
    colorlinks=true,
    linkcolor=black,
    citecolor=black,
    filecolor=black,
    urlcolor=black,
}

\newcommand{\rn}{{\mathbb{R}^n}}

\newcommand{\w}{\omega}
\newcommand{\W}{\Omega}
\newcommand{\eps}{\varepsilon}

\newcommand{\sgn}{\operatorname{sgn}}

\def\R{{\mathbb {R}}}
\def\N{{\mathbb {N}}}

\newtheorem{theorem}{Theorem}[section]
\newtheorem*{theorem*}{Theorem}
\newtheorem{lemma}[theorem]{Lemma}

\newtheorem{corollary}[theorem]{Corollary}

\theoremstyle{remark}
\newtheorem{remark}[theorem]{Remark}

\theoremstyle{definition}
\newtheorem{definition}[theorem]{Definition}

\numberwithin{equation}{section}
\numberwithin{table}{section}
\numberwithin{figure}{section}

\title[Fractional Laplacian]{Regularity Theory and High Order Numerical Methods for the (1D)-Fractional Laplacian}
\author[Acosta,  Borthagaray, Bruno,  Maas]{G. Acosta,  J. P. Borthagaray, O. Bruno and M. Maas}

\address[Gabriel Acosta and Juan Pablo Borthagaray]{IMAS - CONICET and Departamento de Matem\'a\-tica, FCEyN - Universidad de Buenos Aires, Ciudad Universitaria, Pabell\'on I  (1428) Buenos Aires, Argentina.}
\address[Oscar Bruno]{California Institute of Technology. Pasadena, California.}
\address[Mart\'in Maas]{IAFE - CONICET and Departamento de Matem\'a\-tica, FCEyN - Universidad de Buenos Aires, Ciudad Universitaria, Pabell\'on I  (1428) Buenos Aires, Argentina.}

\email[G. Acosta]{gacosta@dm.uba.ar}
\email[J. P. Borthagaray]{jpbortha@dm.uba.ar}
\email[O. Bruno]{obruno@caltech.edu}
\email[M. Maas]{mdmaas@iafe.uba.ar}

\thanks{This research has been partially supported by CONICET under grant PIP 2014-2016 11220130100184CO}

\begin{document}


\begin{abstract}
  This paper presents regularity results and associated high-order
  numerical methods for one-dimensional Fractional-Laplacian
  boundary-value problems.  On the basis of a factorization of
  solutions as a product of a certain edge-singular weight $\w$ times
  a ``regular'' unknown, a characterization of the regularity of
  solutions is obtained in terms of the smoothness of the
  corresponding right-hand sides.  In particular, for right-hand sides
  which are analytic in a Bernstein Ellipse, analyticity in the same
  Bernstein Ellipse is obtained for the ``regular'' unknown.
  Moreover, a sharp Sobolev regularity result is presented which
  completely characterizes the co-domain of the Fractional-Laplacian
  operator in terms of certain weighted Sobolev spaces introduced in
  (Babu\v{s}ka and Guo, SIAM J. Numer. Anal. 2002). The present
  theoretical treatment relies on a full eigendecomposition for a
  certain weighted integral operator in terms of the Gegenbauer
  polynomial basis.  The proposed Gegenbauer-based Nystr\"om numerical
  method for the Fractional-Laplacian Dirichlet problem, further, is
  significantly more accurate and efficient than other algorithms
  considered previously. The sharp error estimates presented in this
  paper indicate that the proposed algorithm is spectrally accurate,
  with convergence rates that only depend on the smoothness of the
  right-hand side. In particular, convergence is exponentially fast
  (resp. faster than any power of the mesh-size) for analytic
  (resp. infinitely smooth) right-hand sides. The properties of the
  algorithm are illustrated with a variety of numerical results.
\end{abstract}

\maketitle
\section{Introduction} \label{sec:intro} Over the last few years
nonlocal models have increasingly impacted upon a number of important
fields in science and technology.
The evidence of anomalous diffusion processes, for example, has been
found in several physical and social environments \cite{Klafter,
  MetzlerKlafter}, and corresponding transport models have been
proposed in various areas such as electrodiffusion in nerve
cells~\cite{AnomalousElectrodiffusion} and ground-water solute
transport~\cite{BensonWheatcraft}.  Non-local models have also been
proposed in fields such as finance~\cite{CarrHelyette,RamaTankov} and
image processing~\cite{GattoHesthaven,GilboaOsher}.
One of the fundamental non-local operators is the Fractional Laplacian
$(-\Delta)^s$ ($0 < s < 1$) which, from a probabilistic point of view
corresponds to the infinitesimal generator of a stable L\'evy process
\cite{Valdinoci}.

The present contribution addresses theoretical questions and puts
forth numerical algorithms for the numerical solution of the Dirichlet
problem
\begin{equation}
\left\lbrace
  \begin{array}{rl}
      (-\Delta)^s u = f & \mbox{ in }\W, \\
      u = 0 & \mbox{ in }\W^c  \\
      \end{array}
    \right.
\label{eq:fraccionario_dirichlet}
\end{equation}
on a bounded one-dimensional domain $\Omega$ consisting of a union of
a finite number of intervals (whose closures are assumed mutually
disjoint).  This approach to enforcement of (nonlocal) boundary
conditions in a bounded domain $\W$ arises naturally in connection
with the long jump random walk approach to the Fractional
Laplacian~\cite{Valdinoci}. In such random walk processes, jumps of
arbitrarily long distances are allowed. Thus, the payoff of the
process, which corresponds to the boundary datum of the Dirichlet
problem, needs to be prescribed in $\W^c$.

Letting $s$ and $n$ denote a real number ($0<s<1$) 
and the spatial dimension ($n=1$ throughout this paper), and
using the normalization constant~\cite{Hitchhikers}
$$C_n(s) = \frac{2^{2s} s \Gamma(s+\frac{n}{2})}{\pi^{n/2} \Gamma(1-s)},  $$
the fractional-Laplacian operator $(-\Delta)^s$ is given by
\begin{equation}
(-\Delta)^s u (x) = C_n(s) \mbox{ P.V.} \int_\rn \frac{u(x)-u(y)}{|x-y|^{n+2s}} \, dy.
\label{eq:fraccionarioyo}
\end{equation}

\begin{remark}\label{rem_1}
  A number of related operators have been considered in the
  mathematical literature.  Here we mention the so called
  \emph{spectral} fractional Laplacian ${\mathcal L}_s$, which is
  defined in terms of eigenfunctions and eigenvalues $(v_n,\lambda_n)$
  {\em of the standard Laplacian ($-\Delta$) operator with Dirichlet
    boundary conditions in $\partial \W$}: ${\mathcal L}_s[v_n] =
  \lambda_n^s v_n$.  The operator ${\mathcal L}_s$ is different from
  $(-\Delta)^s$---since, for example, ${\mathcal L}_s$ admits smooth
  eigenfunctions (at least for smooth domains) $\Omega$ while
  $(-\Delta)^s$ does not; see~\cite{ServadeiValdinoci}.
\end{remark}

\begin{remark}
  A finite element approach for problems concerning the operator
  ${\mathcal L}_s$ (cf. Remark~\ref{rem_1}) was proposed
  in~\cite{NochettoOtarola} on the basis of extension ideas first
  introduced in~\cite{CaffarelliSilvestre} for the operator
  $(-\Delta)^s$ in $\mathbb{R}^n$ which were subsequently developed
  in~\cite{ConcaveConvex} for the bounded-domain operator ${\mathcal
    L}_s$.  As far as we know, however, approaches based on extension
  theorems have not as yet been proposed for the Dirichlet
  problem~\eqref{eq:fraccionario_dirichlet}.
\end{remark}

Various numerical methods have been proposed recently for equations
associated with the Fractional Laplacian $(-\Delta)^s$ in bounded
domains.  Restricting attention to one-dimensional problems, Huang and
Oberman \cite{HuangOberman} presented a numerical algorithm that
combines finite differences with a quadrature rule in an unbounded
domain. Numerical evidence provided in that paper for smooth
right-hand sides (cf. Figure 7(b) therein) indicates convergence to
solutions of~\eqref{eq:fraccionario_dirichlet} with an order
$\mathcal{O}(h^s)$, in the infinity norm, as the meshsize $h$ tends to
zero (albeith orders as high as $\mathcal{O}(h^{3-2s})$ are
demonstrated in that contribution for singular right-hand sides $f$
that make the solution $u$ smooth).  Since the order $s$ lies between
zero and one, the $\mathcal{O}(h^s)$ convergence provided by this
algorithm can be quite slow, specially for small values of $s$.
D'Elia and Gunzburger~\cite{DEliaGunzburger}, in turn, proved
convergence of order $h^{1/2}$ for a finite-element solution of an
associated one-dimensional nonlocal operator that approximates the
one-dimensional fractional Laplacian. These authors also suggested
that an improved solution algorithm, with increased convergence order,
might require explicit consideration of the solution's boundary
singularities.
The contribution~\cite{AcostaBorthagaray}, finally, studies the
regularity of solutions of the Dirichlet
problem~\eqref{eq:fraccionario_dirichlet} and it introduces certain graded
meshes for integration in one- and two-dimensional domains. 
The rigorous error bounds and numerical
experiments provided in~\cite{AcostaBorthagaray} demonstrate an
accuracy of the order of $h^{1/2} |\log h|$ and $h|\log h|$ for all $s$, in
certain weighted Sobolev norms, for solutions obtained by means of
uniform and graded meshes, respectively.

%
%

%
Difficulties in the numerical treatment of the Dirichlet
problem~\eqref{eq:fraccionario_dirichlet} stem mainly from the
singular character of the solutions of this problem near boundaries. A
recent regularity result in this regards was provided
in~\cite{RosOtonSerra}.  In particular, this contribution establishes
the global H\"older regularity of solutions of the general
$n$-dimensional version of equation~\eqref{eq:fraccionario_dirichlet}
($n\ge 1$) and it provides a certain boundary regularity result: the
quotient $u(x)/\w^s(x)$ remains bounded as $x\to \partial \Omega$,
where $\omega$ is a smooth function that behaves like
$\mbox{dist}(x,\Omega^c)$ near $\partial \W$. This result was then
generalized in~\cite{Grubb}, where, using pseudo-differential
calculus, a certain regularity result is established in terms of
H\"ormander $\mu$-spaces: in particular, for the regular Sobolev
spaces $H^r(\W)$, it is shown that if $f\in H^r(\W)$ for some $r>0$
then the solution $u$ may be written as $w^s\phi + \chi$, where $\phi
\in H^{r+s}(\W)$ and $\chi \in H^{r+2s}_0(\W)$. Interior regularity
results for the Fractional Laplacian and related operators have also
been the object of recent studies~\cite{Albanese2015,Cozzi2016}.


The sharp regularity results put forth in the present contribution, in
turn, are related to but different from those mentioned above. Indeed
the present regularity theorems show that the fractional Laplacian in
fact induces a {\em bijection} between certain weighted Sobolev
spaces. Using an appropriate version of the Sobolev lemma put forth in
Section~\ref{regularity}, these results imply, in particular, that the
regular factors in the decompositions of fractional Laplacian solutions
admit $k$ continuous derivatives for a certain value of $k$ that
depends on the regularity of the right-hand side. Additionally, this
paper establishes the operator regularity in spaces of analytic
functions: denoting by $A_\rho$ the space of analytic functions in the
Bernstein Ellipse $\mathcal{E}_\rho$, the weighted operator $K_s(\phi)
= (-\Delta)^s(\w^s \phi)$ maps $A_\rho$ into itself bijectively.  In
other words, for a right-hand side which is analytic in a Bernstein
Ellipse, the solution is characterized as the product of an analytic
function in the same Bernstein Ellipse times an explicit singular
weight.

The theoretical treatment presented in this paper is essentially
self-contained. This approach recasts the problem as an integral
equation in a bounded domain, and it proceeds by computing certain
singular exponents $\alpha$ that make $(-\Delta)^s (\omega^\alpha
\phi(x))$ analytic near the boundary for every polynomial $\phi$. As
shown in Theorem~\ref{teo1} a infinite sequence of such values of
$\alpha$ is given by $\alpha_n = s + n$ for all $n\geq 0$. Morever,
Section~\ref{two_edge_sing} shows that the weighted operator $K_s$
maps polynomials of degree $n$ into polynomials of degree $n$---and it
provides explicit closed-form expressions for the images of each
polynomial $\phi$.
 
A certain hypersingular form we present for the operator $K_s$ leads
to consideration of a weighted $L^2$ space wherein $K_s$ is
self-adjoint.  In view of the aforementioned polynomial-mapping
properties of the operator $K_s$ it follows that this operator is
diagonal in a basis of orthogonal polynomials with respect to a
corresponding inner product.  A related diagonal form was obtained in
the recent independent contribution~\cite{Dyda2016} by employing
arguments based on Mellin transforms. The diagonal
form~\cite{Dyda2016} provides, in particular, a family of explicit
solutions in the $n$ dimensional ball in $\R^n$, which are given by
products of the singular term $(1-|z|^2)^s$ and general Meijer
G-Functions. The diagonalization approach proposed in this paper,
which is restricted to the one-dimensional case, is elementary and is
succinctly expressed: the eigenfunctions are precisely the Gegenbauer
polynomials.

This paper is organized as follows: Section~\ref{integraleq} casts the
problem as an integral equation, and Section~\ref{diagonalform}
analyzes the boundary singularity and produces a diagonal form for the
single-interval problem.  Relying on the Gegenbauer eigenfunctions and
associated expansions found in Section~\ref{diagonalform},
Section~\ref{regularity} presents the aforementioned Sobolev and
analytic regularity results for the solution $u$, and it includes a
weighted-space version of the Sobolev lemma. Similarly, utilizing
Gegenbauer expansions in conjunction with Nystr\"om discretizations
and taking into account the analytic structure of the edge
singularity, Section~\ref{HONM} presents a highly accurate and
efficient numerical solver for Fractional-Laplacian equations posed on
a union of finitely many one-dimensional intervals. The sharp error
estimates presented in Section~\ref{HONM} indicate that the proposed
algorithm is spectrally accurate, with convergence rates that only
depend on the smoothness of the right-hand side. In particular,
convergence is exponentially fast (resp. faster than any power of the
mesh-size) for analytic (resp. infinitely smooth) right-hand sides.  A
variety of numerical results presented in Section~\ref{num_res}
demonstrate the character of the proposed solver: the new algorithm is
significantly more accurate and efficient than those resulting from
previous approaches.



\section{Hypersingular Bounded-Domain  Formulation\label{integraleq}}

In this section the one-dimensional operator 
\begin{equation}
(-\Delta)^s u(x) = C_1(s) \mbox{ P.V.} \int_{-\infty}^\infty \left( u(x)-u(x-y) \right) |y|^{-1-2s} dy
\label{frac1d}
\end{equation} 
together with Dirichlet boundary conditions outside the bounded domain
$\W$, is expressed as an integral over $\W$. The Dirichlet
problem~\eqref{eq:fraccionario_dirichlet} is then identified with a
hypersingular version of Symm's integral equation; the precise
statement is provided in Lemma~\ref{lemma_hypersingular} below.  In
accordance with Section~\ref{sec:intro}, throughout this paper we
assume the following definition holds.
\begin{definition}\label{union_intervals_def}
  The domain $\W$ equals a finite union
\begin{equation}
\label{union_intervals}
\W = \bigcup_{i=1}^M (a_i, b_i)
\end{equation}
of open intervals $(a_i,b_i)$ with disjoint closures. We denote
$\partial\W =\{a_1,b_1,\dots, a_M,b_M\}$.
\end{definition}

\begin{definition}
  $C^2_0(\W)$ will denote, for a given open set $\W \subset
  \mathbb{R}$, the space of all functions $u \in C^2(\W) \cap
  C(\mathbb{R})$ that vanish outside of $\W$. For $\W =(a,b)$ we will
  simply write $C^2_0((a,b)) = C^2_0(a,b)$.
\end{definition}
The following lemma provides a useful expression for the Fractional
Laplacian operator in terms of a certain integro-differential
operator. For clarity the result is first presented in the following
lemma for the case $\W=(a,b)$; the generalization to 
domains $\W$ of the form~\eqref{union_intervals} then follows easily in
Corollary~\ref{coro_lemma_hypersingular}.

\begin{lemma}
\label{lemma_hypersingular}
Let $s \in (0,1)$, let $u \in
C^2_0(a,b)$ such that $|u'|$ is integrable in $(a,b)$, let $x \in \mathbb{R}, x\not \in \partial \W=\{a,b\}$,
and define
\begin{equation}\label{eq:c_s}
C_s = \frac{C_1(s)}{2s(1-2s)} = -\Gamma(2s-1)\sin(\pi s) / \pi \quad (s\ne 1/2);
\end{equation}
We then have
\begin{itemize}
\item [---] Case $s \ne \frac{1}{2}$:
\begin{equation}
\label{eq:hypersingular}
(-\Delta)^s u (x) = C_s \frac{d}{dx} \int_{a}^b |x-y|^{1-2s} \frac{d}{dy} u(y) dy.
\end{equation}
\item [---] Case $s=\frac{1}{2}$:
\begin{equation}
\label{eq:hypersingular_s12}
(-\Delta)^{1/2} u (x) = \frac{1}{\pi} \frac{d}{dx} \int_{a}^b \ln  |x-y| \frac{d}{dy} u(y) dy.
\end{equation}
\end{itemize}
\end{lemma}
\begin{proof}
  We note that, since the support of $u=u(x)$ is contained in $[a,b]$,
  for each $x \in \mathbb{R}$ the support of the translated function
  $u=u(x-y)$ as a function of $y$ is contained in the set
  $[x-b,x-a]$. Thus, using the decomposition $\R = [x-b,x-a] \cup
  (-\infty,x-b) \cup (x-a,\infty)$ in~\eqref{frac1d}, we obtain the
  following expression for $(-\Delta)^s u (x)$:
\begin{equation}\label{spliting}
  C_1(s) \Bigg( \mbox{P.V.} \int_{x-b}^{x-a} ( u(x)-u(x-y) ) |y|^{-1-2s} dy + 
  \left[\int_{-\infty}^{x-b}  dy + \int_{x-a}^\infty dy\right] u(x) |y|^{-1-2s} \Bigg). 
\end{equation}

We consider first the case $x\not\in [a,b]$, for
which~\eqref{spliting} becomes
\begin{equation}\label{spliting_2}
  -C_1(s)\Bigg( \mbox{P.V.} \int_{x-b}^{x-a} u(x-y) |y|^{-1-2s} dy\Bigg). 
\end{equation}
Noting that the integrand~\eqref{spliting_2} is smooth, integration by
parts yields
\begin{equation}\label{parts_firstterm}
  \frac{C_1(s)}{2s} \int_{x-b}^{x-a} u'(x-y) \sgn(y)|y|^{-2s} dy
\end{equation}
(since $u(a)=u(b)=0$), and, thus, letting $z=x-y$ we obtain
\begin{equation}\label{singleinterval_smooth}
  (-\Delta)^s u(x) = \frac{C_1(s)}{2s} \int_{a}^{b} \sgn(x-z)|x-z|^{-2s} u'(z) dz\quad , \quad x\not\in [a,b].
\end{equation}
Then, letting
\begin{equation*}\label{K_def}
\Phi_s(y) = \left\lbrace
  \begin{array}{ll}
    |y|^{1-2s}/(1-2s)  & \mbox{ for } s\in (0,1), \, s\ne 1/2 \\
    \log|y| & \mbox{ for } s = 1/2 
      \end{array}
    \right. ,
\end{equation*}
noting that
\begin{equation}
\label{sgn_x_der}
\sgn(x-z)|x-z|^{-2s} =  
\frac{\partial}{\partial x} \Phi_s(x-z),
\end{equation}
replacing~\eqref{sgn_x_der} in ~\eqref{singleinterval_smooth} and
exchanging the $x$-differentiation and $z$-integration yields the
desired expressions~\eqref{eq:hypersingular}
and~\eqref{eq:hypersingular_s12}. This completes the proof in the case
$x\not\in [a,b]$.

Let us now consider the case $x\in(a,b)$.  The second term
in~\eqref{spliting} can be computed exactly; we clearly have
\begin{equation}\label{eq:bnd_term}
\left[\int_{-\infty}^{x-b} dy + \int_{x-a}^\infty dy\right] u(x) |y|^{-1-2s} = 
\left[ \frac{u(x)}{2s} \sgn(y) |y|^{-2s} \bigg|_{y=x-b}^{y=x-a} \right] .
\end{equation}
In order to integrate by parts in the P.V. integral in~\eqref{spliting} consider the set $$ D_\eps =
[x-b,x-a] \setminus (-\eps,\eps).$$
Then, defining
\begin{equation*}
\label{parts_secondterm_0}
Q_\epsilon (x) = \int_{D_\epsilon } \left( u(x)-u(x-y) \right) |y|^{-1-2s} dy 
\end{equation*}
integration by parts yields
\begin{equation*}
\label{parts_secondterm}
Q_\epsilon  (x) = 
-\frac{1}{2s}\left ( g_{a}^b(x) - h_{a}^b(x) - \frac{\delta^2_\eps}{\eps^{2s}} -  \int_{D_\epsilon} u'(x-y) \sgn(y)|y|^{-2s} dy \right)
\end{equation*}
where $\delta_\eps = u(x+\eps) + u(x-\eps) - 2 u(x)$, $g_{a}^b(x) =
u(x)(|x-a|^{-2s} + |x-b|^{-2s})$ and $h_{a}^b(x) = u(a)|x-a|^{-2s} +
u(b)|x-b|^{-2s}$. 

The term $h_{a}^b(x)$ vanishes since $u(a)=u(b)=0$. The contribution
$g_{a}^b(x)$, on the other hand, exactly cancels the boundary terms in
equation~\eqref{eq:bnd_term}.
%
%
%
%
For the values $x\in (a,b)$ under
consideration, a Taylor expansion in $\eps$ around $\eps=0$
additionally tells us that the quotient
$\frac{\delta^2_\eps}{\eps^{2s}}$ tends to $0$ as $\eps\to
0$. Therefore, using the change of variables $z=x-y$ and letting
$\eps\to 0$ we obtain a principal-value expression valid for $x \ne a, x\ne b$:
\begin{equation}
\label{eq:pv_singleinterval}
(-\Delta)^s u (x) = \frac{C_1(s)}{2s} \mbox{ P.V.} \int_{a}^b \sgn(x-z)|x-z|^{-2s} u'(z) dz.
\end{equation}
Replacing~\eqref{sgn_x_der} in ~\eqref{eq:pv_singleinterval} then
yields~\eqref{eq:hypersingular} and~\eqref{eq:hypersingular_s12},
provided that the derivative in $x$ can be interchanged with the
P.V. integral. This interchange is indeed correct, as it follows from
an application of the following Lemma to the function $v=u'$. The
proof is thus complete.
\end{proof}
\begin{lemma}
\label{lemma_exchangePV}
Let $\Omega \subset \mathbb{R}$ be as indicated in
Definition~\ref{union_intervals_def} and let $v\in C^1(\Omega)$ such
that $v$ is absolutely integrable over $\Omega$, and let
$x\in\Omega$. Then the following relation holds:
\begin{equation} \label{der_pv}
 P.V. \int_\Omega \frac{\partial}{\partial x} \Phi_s(x-y) v(y) dy = \frac{\partial}{\partial x} \int_\Omega \Phi_s(x-y) v(y) dy
\end{equation}
\end{lemma}
\begin{proof}
See Appendix~\ref{app_exchangePV}.
\end{proof}

\begin{corollary}
\label{coro_lemma_hypersingular}
Given a domain $\W$ as in Definition~\eqref{union_intervals_def}, and
with reference to equation~\eqref{eq:c_s}, for $u\in C_0^2(\W)$ and
$x\not\in \partial \Omega$ we have
\begin{itemize}
\item [---] Case $s \ne \frac{1}{2}$:
\begin{equation}
\label{eq:mutli_int}
(-\Delta)^s u (x) = C_s \frac{d}{dx} \sum_{i=1}^M \int_{a_i}^{b_i} |x-y|^{1-2s} \frac{d}{dy} u(y) dy 
\end{equation}
\item [---] Case $s=\frac{1}{2}$:
\begin{equation}
\label{eq:mutli_int_s12}
 (-\Delta)^{1/2} u (x) = \frac{1}{\pi} \frac{d}{dx} \sum_{i=1}^M \int_{a_i}^{b_i} \ln  |x-y| \frac{d}{dy} u(y) dy 
\end{equation}
\end{itemize}
for all $x\in\mathbb{R}\setminus \partial \Omega = \cup_i^M \{a_i, b_i\}$.
\end{corollary}
\begin{proof}
  Given $u\in C_0^2(\W)$ we may write $u=\sum_i^M u_i$ where, for
  $i=1,\dots,M$ the function $u_i=u_i(x)$ equals $u(x)$ for $x \in
  (a_i,b_i)$ and and it equals zero elsewhere. In view of
  Lemma~\ref{lemma_hypersingular} the result is valid for each
  function $u_i$ and, by linearity, it is thus valid for the function
  $u$. The proof is complete.
\end{proof}
\begin{remark}\label{bound_rem}
  A point of particular interest arises as we examine the character of
  $(-\Delta)^s u$ with $u\in C_0^2(\W)$ for $x$ at or near $\partial
  \W$. Both Lemma~\ref{lemma_hypersingular} and its
  corollary~\ref{coro_lemma_hypersingular} are silent in these
  regards. For $\W = (a,b)$, for example, inspection of
  equation~\eqref{eq:pv_singleinterval} leads one to generally expect
  that $(-\Delta)^s u(x)$ has an infinite limit as $x$ tends to each
  one of the endpoints $a$ or $b$. But this is not so for all
  functions $u\in C_0^2(\W)$. Indeed, as established in
  Section~\ref{diag}, the subclass of functions in $C_0^2(\W)$ for
  which there is a finite limit forms a dense subspace of a relevant
  weighted $L^2$ space.  In fact, a dense subset of functions exists
  for which the image of the fractional Laplacian can be extended as
  an analytic function in the complete complex $x$ variable plane.
  But, even for such functions, definition~\eqref{frac1d} still
  generically gives $(-\Delta)^s u (x) =\pm \infty$ for $x=a$ and
  $x=b$. Results concerning functions whose Fractional Laplacian blows
  up at the boundary can be found in~\cite{Abatangelo2015}.
\end{remark}
The next section concerns the single-interval case ($M=1$
in~\eqref{eq:mutli_int},~\eqref{eq:mutli_int_s12}).  Using
translations and dilations the single interval problem in any given
interval $(a_1, b_1)$ can be recast as a corresponding problem in any
desired open interval $(a,b)$. For notational convenience two
different selections are made at various points in
Section~\ref{diagonalform}, namely $(a,b)=(0,1)$ in
Sections~\ref{edge_sing_left} and~\ref{two_edge_sing}, and $(a,
b)=(-1,1)$ in Section~\ref{diag}.  The conclusions and results can
then be easily translated into corresponding results for general
intervals; see for example Corollary~\ref{diag_ab}.

\section{Boundary Singularity and Diagonal Form of the Single-Interval
  Operator \label{diagonalform}}

Lemma~\ref{lemma_hypersingular} expresses the action of the operator
$(-\Delta)^s$ on elements $u$ of the space $C_0^2(\W)$ in terms of the
integro-differential operators on the right-hand side of
equations~\eqref{eq:hypersingular} and~\eqref{eq:hypersingular_s12}. A
brief consideration of the proof of that lemma shows that for such
representations to be valid it is essential for the function $u$ to
vanish on the boundary---as all functions in $C_0^2(a,b)$ do, by
definition. Section~\ref{edge_sing_left} considers, however, the
action under the integral operators on the right-hand side of
equations~\eqref{eq:hypersingular} and~\eqref{eq:hypersingular_s12} on
certain functions $u$ defined on $\W = (a,b)$ {\em which do not
  necessarily vanish at $a$ or $b$}. To do this
we study the closely related integral operators
 \begin{align}
   S_s[u](x) &:= C_s \int_a^b \left( |x-y|^{1-2s} - (b-a)^{1-2s} \right) u(y) dy \;\;  ( s \ne \frac{1}{2} ), \label{Tsdef_eq1}\\
   S_{\frac{1}{2}}[u](x) &:= \frac{1}{\pi} \int_a^b \log\left(\frac{|x-y|}{b-a}\right) u(y) dy,  \label{Tsdef_eq2}\\
   T_s[u](x) &:=  \frac{\partial}{\partial x} S_s\left[
     \frac{\partial}{\partial y} u(y) \right](x) .\label{Tsdef_eq3}
\end{align}

\begin{remark}\label{const_term}
  The addition of the constant term $-(b-a)^{1-2s}$ in the
  integrand~\eqref{Tsdef_eq1} does not have any effect in the
  definition of $T_s$: the constant $-(b-a)^{1-2s}$ only results in the addition
  of a constant term on the right-hand side of~\eqref{Tsdef_eq1},
  which then yields zero upon the outer differentiation in
  equation~\eqref{Tsdef_eq3}. The integrand~\eqref{Tsdef_eq1} is
  selected, however, in order to insure that the kernel of $S_s$
  (namely, the function $C_s \left( |x-y|^{1-2s} -(b-a)^{1-2s}\right)$) tends to
  the kernel of $S_{\frac{1}{2}}$ in~\eqref{Tsdef_eq2} (the function
  $\frac{1}{\pi} \log(|x-y|/(b-a))$) in the limit as $s\to \frac{1}{2}$.
\end{remark}
\begin{remark}\label{Ts_PV}
  In view of Remark~\ref{const_term} and Lemma~\ref{lemma_exchangePV},
  for $u\in C^2(a,b)$ we additionally have
\begin{equation}\label{Tsdef2}
  T_s[u](x) = \frac{C_1(s)}{2s} \mbox{ P.V.} \int_{a}^b
  \sgn(x-z)|x-z|^{-2s} u'(z) dz.
\end{equation}
\end{remark}
\begin{remark}
\label{remark_Ts}
  The operator $T_s$ coincides with $(-\Delta)^s$ for functions $u$
  that satisfy the hypothesis of Lemma~\ref{lemma_hypersingular}, but
  $T_s$ does not coincide with $(-\Delta)^s$ for functions $u$ which,
  such as those we consider in Section~\ref{edge_sing_left} below, do
  not vanish on $\partial \W = \{a,b\}$.
\end{remark}

\begin{remark}\label{remark_openarcs}
  The operator $S_{\frac{1}{2}}$ coincides with Symm's integral
  operator~\cite{Symm}, which is important in the context of
  electrostatics and acoustics in cases where Dirichlet boundary
  conditions are posed on infinitely-thin open
  plates~\cite{OpenArcsRadioScience,OpenArcsTheoretical,Symm,YanSloan}. The
  operator $T_{\frac{1}{2}}$, on the other hand, which may be viewed
  as a {\em hypersingular version} of the Symms operator
  $S_{\frac{1}{2}}$, similarly relates to electrostatics and
  acoustics, in cases leading to Neumann boundary conditions posed on
  open-plate geometries. The operators $S_{s}$ and $T_{s}$ in the
  cases $s \ne \frac{1}{2}$ can thus be interpreted as generalizations
  to fractional powers of classical operators in potential theory,
  cf. also Remark~\ref{remark_Ts}.
\end{remark}
Restricting attention to $\W = (a,b) = (0,1)$ for notational
convenience and without loss of generality,
Section~\ref{edge_sing_left} studies the image $T_s[u_\alpha]$ of the
function 
\begin{equation}\label{u_alpha}
  u_\alpha(y) =y^\alpha
\end{equation}
with $\Re\alpha > 0$---which is smooth in $(0,1)$, but which has an
algebraic singularity at the boundary point $y=0$. That section shows
in particular that, whenever $\alpha = s +n$ for some $n\in
\mathbb{N}\cup \{ 0 \}$, the function $T_s[u_\alpha](x)$ can be
extended analytically to a region containing the boundary point $x=0$.
Building upon this result (and assuming once again $\W = (a, b) =
(0,1)$), Section~\ref{two_edge_sing}, explicitly evaluates the images
of functions of the form $v(y) = y^{s+n}(1-y)^s$ ($n\in \mathbb{N}\cup
\{ 0 \}$), which are singular (not smooth) at the two boundary points
$y=0$ and $y=1$, under the integral operators $T_s$ and $S_s$. The
results in Section~\ref{two_edge_sing} imply, in particular, that the
image $T_s[v]$ for such functions $v$ can be extended analytically to
a region containing the interval $[0,1]$.  Reformulating all of these
results in the general interval $\W = (a,b)$, Section~\ref{diag} then
derives the corresponding single-interval diagonal form for weighted
operators naturally induced by $T_s$ and $S_s$.
\subsection{Single-edge singularity\label{edge_sing_left}}
With reference to equations~\eqref{Tsdef2} and \eqref{eq:c_s}, and
considering the aforementioned function $u_\alpha(y) = y^\alpha$ we
clearly have
\begin{equation*}\label{eq:TsNs}
  T_s[u_\alpha](x) =\alpha (1-2s) C_s N_\alpha^s(x) \;\; \mbox{, where}
\end{equation*}
\begin{equation}
  N_\alpha^s(x) := P.V.\int_{0}^{1} \sgn(x-y)|x-y|^{-2s} y^{\alpha-1} dy .
\label{eq:def_Ns}
\end{equation} 
As shown in Theorem~\ref{teo1} below (equation~\eqref{eq:Ns_Betas}),
the functions $N_\alpha^s$ and (thus) $ T_s[u_\alpha]$ can be
expressed in terms of classical special functions whose singular
structure is well known. Leading to the proof of that theorem, in what
follows we present a sequence of two auxiliary lemmas.
\begin{lemma}
\label{lemma0_analytic}
Let $ E = (a, b)\subset\mathbb{R}$, and let $C\subseteq \mathbb{C}$
denote an open subset of the complex plane. Further, let $f=f(t,c)$ be
a function defined in $E\times C$, and assume 1)~$f$ is continuous in
$E\times C$, 2)~$f$ is analytic with respect to $c=c_1+ic_2\in C$ for
each fixed $t\in E$, and 3)~$f$ is ``uniformly integrable over
compact subsets of $C$''---in the sense that for every compact set $K
\subset C$ the functions
\begin{equation}
\label{hab_eta}
h_a(\eta,c) = \left| \int_a^{a+\eta} f(t,c) dt \,\right|\quad \mbox{and}\quad
h_b(\eta,c) = \left| \int_{b-\eta}^b f(t,c) dt \,\right|
\end{equation}
tend to zero uniformly for $c\in K$ as $\eta\to 0^+$.  Then the
function
$$F(c) := \int_E f(t,c) dt$$
is analytic throughout $C$.
\end{lemma}
\begin{proof} 
  Let $K$ denote a compact subset of $C$. For each $c\in K$ and each
  $n\in\mathbb{N}$ we consider Riemann sums $R_n^h(c)$ for the
  integral of $f$ in the interval $[a+\eta_n,b-\eta_n]$, where
  $\eta_n$ is selected in such a way that $h_a(\eta_n,c) \leq 1/n$ and
  $h_b(\eta_n,c) \leq 1/n$ for all $c\in K$ (which is clearly possible
  in view of the hypothesis~\eqref{hab_eta}). The Riemann sums are
  defined by $R_n^h(c)=h \sum_{j=1}^Mf(t_j,c)$, with $ h = (b-a
  +2\eta_n)/M$ and $t_{j+1} - t_j =h $ for all $j$.

  Let $n\in \mathbb{N}$ be given. In view of the uniform continuity of
  $f(t,c)$ in the compact set $[a+\eta_n,b-\eta_n]\times K$, the
  difference between the maximum and minimum of $f(t,c)$ in each
  integration subinterval $(t_j,t_{j+1})\subset [a+\eta_n,b-\eta_n]$
  tends uniformly to zero for all $c\in K$ as the integration meshsize
  tends to zero. It follows that a meshsize $h_n$ can be found for
  which the approximation error in the corresponding Riemann sum
  $R_n^h(c)$ is {\em uniformly small for all $c\in K$}:
\[
\left| \int_{a+\eta_n}^{b-\eta_n} f(t,c) dt - R_n^h(c)\right| < \frac 1n
\quad \mbox{for all $c\in K$ and for all $n\in \mathbb{N}$}.
\]
Thus $F(c)$ equals a uniform limit of analytic functions over every
compact subset of $C$, and therefore $F(c)$ is itself analytic
throughout $C$, as desired.
\end{proof}

\begin{lemma}\label{lemma_analytic}
  Let $x\in (0,1)$ and let $g(s,\alpha) = N_\alpha^s(x)$ be defined
  by~\eqref{eq:def_Ns} for complex values of $s$ and $\alpha$
  satisfying $ \Re s < 1$ and $\Re \alpha > 0$. We then have:
\begin{itemize}
\item[(\it{i})] For each fixed $\alpha$ such that $\Re \alpha > 0$,
  $g(s,\alpha)$ is an analytic function of $s$ for $\Re s < 1$; and
\item[(\it{ii})] For each fixed $s$ such that $\Re s < 1$,
  $g(s,\alpha)$ is an analytic of $\alpha$ for $\Re \alpha > 0$.
\end{itemize}
In other words, for each fixed $x\in (0,1)$ the function $
N_\alpha^s(x)$ is jointly analytic in the $(s,\alpha)$ domain $D = \{
\Re s < 1\} \times \{\Re \alpha > 0\} \subset \mathbb{C}^2$.
\end{lemma}
\begin{proof} 
  We express the integral that defines $N_\alpha^s$ as the sum $g_1(s,\alpha) +
  g_2(s,\alpha)$ of two integrals, each one of which contains only one of the
  two singular points of the integrand ($y=0$ and $y=x$):
\begin{equation*}\label{ns_two}
  g_1 =  \int_{0}^{x/2}
  \sgn(x-y)|x-y|^{-2s} y^{\alpha-1} dy \, \mbox{ and } g_2 = P.V. \int_{x/2}^1
  \sgn(x-y)|x-y|^{-2s} y^{\alpha-1} dy. 
\end{equation*}
Lemma~\ref{lemma0_analytic} tells us that $g_1$ is an analytic
function of $s$ and $\alpha$ for $(s,\alpha)\in D_1 = \mathbb{C}\times
\{\Re \alpha > 0\}$.  

Integration by parts in the $g_2$ term, in turn, yields
\begin{equation}
\label{g2_parts}
(1-2s) g_2(s,\alpha) = (1-x)^{1-2s} -  \left( \frac{x}{2}\right) ^{\alpha-2s}-
(\alpha-1) \int_{x/2}^{1} |x-y|^{1-2s} y^{\alpha-2} dy .
\end{equation}
But, writing the the integral on the right-hand side
of~\eqref{g2_parts} in the form $\int_{x/2}^1 = \int_{x/2}^x +
\int_{x}^1$ and applying Lemma~\ref{lemma0_analytic} to each one of
the resulting integrals shows that the quantity $(1-2s) g_2(s,\alpha)$
is an analytic function of $s$ and $\alpha$ for $(s,\alpha)\in D_2 =
\mathbb{C}\times \{ \alpha > 0\}$. In view of the $(1-2s)$ factor,
however, it still remains to be shown that $g_2(s,\alpha)$ is analytic
at $s=1/2$ as well.

To check that both $g_2(s,\alpha)$ and $g(s,\alpha)$ are analytic
around $s=1/2$ for any fixed $\alpha \in \{ \Re \alpha>0 \}$, we first
note that since $\int_0^1 1\cdot y^{\alpha-1} dy$ is a constant
function of $x$ we may write
$$ g(s,\alpha) = \frac{1}{1-2s} \frac{\partial}{\partial x} \int_0^1 \left(|x-y|^{1-2s} - 1 \right) y^{\alpha-1} dy. $$
But since we have the uniform limit
\[
\lim_{s\to 1/2}\frac{|x-y|^{1-2s} - 1 }{1-2s} =
\left .\frac{\partial}{\partial r} |x-y|^r\right|_{r=0} = \log|x-y|
\]
as complex values of $s$ approach $s=1/2$, we see that $g$ is in fact a
continuous and therefore, by Riemann's theorem on removable
singularities, analytic at $s=1/2$ as well. The proof is now complete.
\end{proof}
%
%
\begin{theorem}\label{teo1} 
  Let $s\in (0,1)$ and $\alpha >0$. Then $N_\alpha^s(x)$ can be
  analytically continued to the unit disc $\{x:|x|<1\}\subset
  \mathbb{C}$ if and only if either $\alpha = s + n$ or $\alpha = 2s +
  n$ for some $n\in \mathbb{N}\cup \{ 0 \}$. In the case $\alpha = s +
  n$, further, we have
  \begin{equation}
  \label{eq:teo1}
  N_{s+n}^s(x) = \sum_{k=0}^{\infty} \frac{(2s)_k}{s-n+k} \frac{x^k}{k!} 
\end{equation}
where, for a given complex number $z$ and a given non-negative integer $k$
\begin{equation}
\label{def_Pochhamer}
(z)_k:=\frac{\Gamma(z+k)}{\Gamma(z)}
  \end{equation}
denotes the Pochhamer symbol.
\end{theorem}
\begin{proof}
  We first assume $s<\frac{1}{2}$ (for which the integrand
  in~\eqref{eq:def_Ns} is an element of $L^1(0,1)$) and $\alpha < 2s$
  (to enable some of the following manipulations); the result for the
  full range of $s$ and $\alpha$ will subsequently be established by
  analytic continuation in these variables. Writing
$$ N_\alpha^s(x) = x^{-2s} \int_{0}^{1} \sgn(x-y) \left|1-\frac{y}{x}\right|^{-2s} y^{\alpha-1} dy ,$$
after a change of variables and some simple calculations for $x\in
(0,1)$ we obtain
\begin{equation}
\label{eq:Ns_secondterm}
N_\alpha^s(x) = x^{-2s+\alpha} \left[ \int_{0}^{1} (1-r)^{-2s} r^{\alpha-1} dr - \int_{1}^{\frac{1}{x}} (r-1)^{-2s} r^{\alpha-1} dr \right].
\end{equation}
It then follows that
\begin{equation}
\label{eq:Ns_Betas}
N_\alpha^s(x) = x^{-2s+\alpha} \left[\mbox{B}(\alpha, 1 - 2 s)
  -\mbox{B}(1 - 2 s,2s-\alpha) + \mbox{B}_x(-\alpha + 2 s, 1 - 2 s)
\right],
\end{equation}
where 
\begin{equation}\label{Betas}
 \begin{split}
\mbox{B}(a,b) & := \int_0^1 t^{a-1}(1-t)^{b-1} dt = \frac{\Gamma(a)\Gamma(b)}{\Gamma(a+b)}  \;\;\; \mbox{ and} \\
\mbox{B}_x(a,b) & := \int_0^x t^{a-1}(1-t)^{b-1} dt = x^{a} \sum_{k=0}^{\infty} \frac{(1-b)_k}{a+k} \frac{x^k}{k!}
 \end{split}
\end{equation}
denote the Beta Function~\cite[eqns. 6.2.2]{AbramowitzStegun} and the
Incomplete Beta function~\cite[eqns. 6.6.8 and
15.1.1]{AbramowitzStegun}, respectively.  Indeed, the first integral
in~\eqref{eq:Ns_secondterm} equals the first Beta function on the
right-hand side of~\eqref{eq:Ns_Betas}, and, after the change of
variables $w= 1/r$, the second integral is easily seen to equal the
difference $\mbox{B}(1-2s,2s-\alpha) - \mbox{B}_x(-\alpha + 2 s, 1 - 2
s)$.

In view of~\eqref{eq:Ns_Betas} and the right-hand expressions in
equation~\eqref{Betas} we can now write
\begin{equation}
\label{eq:caso_s_unmedio}
 N_\alpha^s(x) = x^{-2s+\alpha} \left[ \frac{ \Gamma(\alpha) \Gamma(1 - 2 s) }{ \Gamma(1+\alpha-2s) } - \frac{ \Gamma(1 - 2 s) \Gamma(-\alpha + 2 s) }{\Gamma(1-\alpha)} \right] + \sum_{k=0}^{\infty} \frac{(2s)_k}{2s-\alpha+k} \frac{x^k}{k!}
\end{equation}
for all $x\in(0,1)$, $0<s<\frac{1}{2}$ and $0<\alpha < 2s$. Using
Euler's reflection formula $\Gamma(z)\Gamma(1-z) = \pi \csc(\pi z)$
(\cite[eq. 6.1.17]{AbramowitzStegun}), and further trigonometric identities,
equation~\eqref{eq:caso_s_unmedio} can also be made to read
\begin{equation}
\label{eq:Ns_euler}
 N_\alpha^s(x) = x^{-2s+\alpha} \frac{\Gamma(\alpha)\Gamma(1 - 2 s)}{\Gamma(1+\alpha-2s)} \frac{2 \cos(\pi s) \sin(\pi (\alpha-s))}{\sin(\pi(\alpha-2s))}  + \sum_{k=0}^{\infty} \frac{(2s)_k}{2s-\alpha+k} \frac{x^k}{k!} .
\end{equation}

The required $x$-analyticity properties of the function
$N_\alpha^s(x)$ will be established by resorting to analytic
continuation of the function $N_\alpha^s(x)$ to complex values of the
variables $s$ and $\alpha$.  In view of the special role played by the
quantity $q = \alpha - 2 s$ in~\eqref{eq:Ns_euler}, further, it is
useful to consider the function $M_{q}^s(x)= N_{q+2s}^s(x)$ where $q$
is defined via the the change of variables $\alpha = q + 2s$.  Then,
collecting for each $n\in \mathbb{N}\cup \{0\}$ all the potentially
singular terms in a neighborhood of $q=n$ and letting $G(s) :=
2\Gamma(1-2s)\cos(\pi s)$ we obtain
\begin{equation} 
\label{eq:Ns_euler2} \begin{split}
  M_{q}^s(x) & = N_{q+2s}^s(x) = \\
	 & = \left[ x^{q} \frac{\Gamma(q+2s) G(s) \sin(\pi (q+s))}{\Gamma(1+q)\sin(\pi q)}  + \frac{(2s)_n}{n-q} \frac{x^n}{n!} \right]  + \sum_{k=0,\; k\ne n}^{\infty} \frac{(2s)_k}{k-q} \frac{x^k}{k!}.
	\end{split}
\end{equation}

In order to obtain expressions for $N_\alpha^s(x)$ which manifestly
display its analytic character with respect to $x$ for all required
values of $s$ and $\alpha$, we analytically continue the function
$M_{q}^s$ to all complex values of $q$ and $s$ for which the
corresponding $(s,\alpha)$ point belongs to the domain $D=\{(s,\alpha)
: \Re s < 1\} \times \{ \Re \alpha>0 \} \subset \mathbb{C}^2$. To do
this we consider the following facts:
\begin{enumerate}
\item \label{dos} Since $\Gamma(z)$ is a never-vanishing function of
  $z$ whose only singularities are simple poles at the nonpositive
  integers $z=-n$ ($n\in \mathbb{N}\cup \{0\}$), and since, as a
  consequence, $1/\Gamma(z)$ is an entire function of $z$ which only
  vanishes at non-positive integer values of $z$, the quotient
  $\Gamma(\alpha) / \Gamma(1+\alpha-2s)$ is analytic and non-zero for
  $(s,\alpha)\in D$.
\item \label{tres} The function $G(s)$ that appears on the right hand 
  side of~\eqref{eq:Ns_euler2} ($s\ne 1/2$) can be continued analytically to the domain $\Re s < 1$ with
  the value $G(1/2)=\pi$. Further, this function does not vanish for
  any $s$ with $0 < \Re s < 1$.
\item \label{cuatro} For fixed $s\in \mathbb{C}$ the quotient
  $\sin(\pi (\alpha-s)) / \sin(\pi (\alpha-2s)) = \sin(\pi (q+s)) /
  \sin(\pi q)$ is a meromorphic function of $q$---whose singularities
  are simple poles at the integer values $q = n\in \mathbb{Z}$ with
  corresponding residues given by $(-1)^n \sin(\pi
  (q+s))/\pi$. Further, for $s\not\in\mathbb{Z}$ the quotient vanishes
  if and only if $q=n-s$ (or equivalently, $\alpha = s + n$) for some $n\in
  \mathbb{Z}$.
\item \label{cinco} For each $x$ in the unit disc $\{x\in\mathbb{C}:
  |x|<1\}$ the infinite series on the right-hand side
  of~\eqref{eq:Ns_euler} converges uniformly over compact subsets of $D
  \setminus \{ \alpha=2s+n, n\in \mathbb{N}\cup \{0\} \}$. This is
  easily checked by using the asymptotic
  relation~\cite[6.1.46]{AbramowitzStegun} $\lim_{k\to \infty}
  k^{1-2s}(2s)_k / k! = 1/\Gamma(2s)$, and taking into account that
  the functions $s\to (2s)_k$ and $s\to 1/\Gamma(2s)$ are entire and,
  thus, finite-valued for each $s\in \mathbb{C}$ and each $k\in
  \mathbb{N}\cup \{0\}$.
\item \label{seis} For each fixed $s\in \mathbb{C}$ and each
  $x\in\mathbb{C}$ with $|x|<1$ the series on the right hand side
  of~\eqref{eq:Ns_euler} is a meromorphic function of $q$ containing
  only simple polar singularities at $q = n\in \mathbb{N}\cup \{0\}$,
  with corresponding residues given by $(2s)_n x^n/n!$. Indeed,
  point~\eqref{cinco} above tells us that the series is an analytic
  function of $q$ for $q \not\in \mathbb{N}\cup \{0\}$; the residue at
  the non-negative integer values of $q$ can be computed immediately
  by considering a single term of the series.
\item \label{siete} The residue of the two terms under brackets on the
  right-hand side of~\eqref{eq:Ns_euler2} are negatives of each
  other. This can be established easily by considering points
  \eqref{cuatro} and~\eqref{seis} as well as the identity $\lim_{q\to
    n} (-1)^n G(s)\sin(\pi(q+s))/\pi = 1/\Gamma(2s)$---which itself
  results from Euler's reflection formula and standard trigonometric
  identities.
\item \label{ocho} The sum of the bracketed terms
  in~\eqref{eq:Ns_euler2} is an analytic function of $q$ up to and
  including non-negative integer values of this variable, as it
  follows from point~\eqref{siete}. Its limit as $q\to n$, further, is
  easily seen to equal the product of an analytic function of $q$ and
  $s$ times the monomial $x^n$.
\end{enumerate}

Expressions establishing the $x$-analyticity properties of
$N_\alpha^s(x)$ can now be obtained. On one hand, by
Lemma~\ref{lemma_analytic} the function $N_\alpha^s(x)$ is a jointly
analytic function of $(s,\alpha)$ in the domain $D$. In view of
points~\eqref{cuatro} through ~\eqref{ocho}, on the other hand, we see
that the right-hand side expression in equation~\eqref{eq:Ns_euler} is
also an analytic function throughout $D$. Since, as shown above in
this proof, these two functions coincide in the open set $U :=
(0,\frac{1}{2}) \times ( 0, 2s )\subset D$, it follows that they must
coincide throughout $D$. In other words, interpreting the right-hand
sides in equations~\eqref{eq:Ns_euler} and~\eqref{eq:Ns_euler2} as
their analytic continuation at all removable-singularity points
(cf. points~\eqref{tres} and~\eqref{siete}) these two equations hold
throughout $D$.

We may now establish the $x$-analyticity of the function
$N_\alpha^s(x)$ for given $\alpha$ and $s$ in $D$.  We first do this
in the case $\alpha = s+n$ with $n\in \mathbb{N}\cup \{0\}$ and
$s\in(0,1)$. Under these conditions the complete first term
in~\eqref{eq:Ns_euler} vanishes---even at $s=1/2$---as it follows from
points~\eqref{dos} through~\eqref{cuatro}. The function
$N_\alpha^s(x)$ then equals the series on the right-hand side
of~\eqref{eq:Ns_euler}. In view of point~\eqref{cinco} we thus see
that, at least in the case $\alpha = s+n$, $N_\alpha^s(x)$ is analytic
with respect to $x$ for $|x|<1$ and, further, that the desired
relation~\eqref{eq:teo1} holds.

In order to establish the $x$-analyticity of $N_\alpha^s(x)$ in the
case $\alpha = 2s+n$ (or, equivalently, $q=n$) with $n\in
\mathbb{N}\cup \{0\}$ and $s\in (0,1)$, in turn, we consider the limit
$q\to n$ of the right-hand side in
equation~\eqref{eq:Ns_euler2}. Evaluating this limit by means of
points~\eqref{cinco} and~\eqref{ocho} results in an expression which,
in view of point~\eqref{cinco}, exhibits the $x$-analyticity of the
function $N_\alpha^s$ for $|x|<1$ in the case under consideration.

To complete our description of the analytic character of
$N_\alpha^s(x)$ for $(\alpha,s)\in D$ it remains to show that this
function is not $x$-analytic near zero whenever
$(\alpha - s)$ and $(\alpha - 2s)$ are not elements of $\mathbb{N}\cup
\{0\}$. But this follows directly by consideration
of~\eqref{eq:Ns_euler}---since, per points~\eqref{dos}, ~\eqref{tres}
and ~\eqref{cuatro}, for such values of $\alpha$ and $s$ the
coefficient multiplying the non-analytic term $x^{-2s+\alpha}$
in~\eqref{eq:Ns_euler} does not vanish. The proof is now complete.
\end{proof}

\subsection{Singularities on both edges\label{two_edge_sing}}
Utilizing Theorem~\ref{teo1}, which in particular establishes that the
image of the function $u_\alpha(y) = y^\alpha$
(equation~\eqref{u_alpha}) under the operator $T_s$ is analytic for
$\alpha = s+n$, here we consider the image of the function
\begin{equation}\label{udef}
 u(y) := y^s(1-y)^s y^n 
\end{equation}
under the operator $T_s$ and we show that, in fact, $T_s[u]$ is a
polynomial of degree $n$. This is a desirable result which, as we
shall see, leads in particular to (i)~Diagonalization of weighted
version of the fractional Laplacian operator, as well as (ii)~Smoothness
and even analyticity (up to a singular multiplicative weight) of
solutions of equation~\eqref{eq:fraccionario_dirichlet} under suitable
hypothesis on the right-hand side $f$.

\begin{remark}\label{remark_idea_2s}
  Theorem~\ref{teo1} states that the image of the aforementioned
  function $u_\alpha$ under the operator $T_s$ is analytic not only
  for $\alpha = s+n$ but also for $\alpha = 2s+n$. But, as shown in
  Remark~\ref{remark_2s}, the smoothness and analyticity theory
  mentioned in point~(ii) above, which applies in the case $\alpha =
  s+n$, cannot be duplicated in the case $\alpha = 2s+n$. Thus, except
  in Remark~\ref{remark_2s}, the case $\alpha = 2s+n$ will not be further
  considered in this paper.
\end{remark}

In view of Remark~\ref{Ts_PV} and in order to obtain an explicit
expression for $T_s[u]$ we first express the derivative of $u$ in the
form
$$u'(y) = \frac{d}{dy} \left(y^{s} (1-y)^s y^n \right) = y^{s-1}
(1-y)^{s-1} \left[ y^n (s+n-(2s+n)y) \right] $$
and (using~\eqref{eq:c_s}) we thus obtain
\begin{equation}
\label{eq:Ts_useful}
 T_s[ u ] = (1-2s)C_s \left( (s+n) L^s_n - (2s+n) L^s_{n+1} \right) .
\end{equation}
where
\begin{equation}
\label{eq:Ks_n}
L^s_n := P.V. \int_{0}^{1} \sgn(x-y)|x-y|^{-2s} y^{s-1} (1-y)^{s-1} y^n dy
\end{equation}
On the other hand, in view of definitions~\eqref{Tsdef_eq1}
and~\eqref{Tsdef_eq2} and Lemma~\ref{lemma_exchangePV} it is easy to
check that
\begin{equation}
\label{eq:Ss_useful}
\frac{\partial}{\partial x} S_{s}( y^{s-1} (1-y)^{s-1} y^n ) = (1-2s)C_s L^s_n .
\end{equation}
In order to characterize the image $T_s[u]$ of the function $u$
in~\eqref{udef} under the operator $T_s$, Lemma~\ref{lemma_Lns} below
presents an explicit expression for the closely related function
$L^s_n$. In particular the lemma shows that $L^s_n$ is a polynomial of
degree $n-1$, which implies that $T_s[u]$ is a polynomial of degree
$n$.
\begin{lemma}
\label{lemma_Lns}
 $L^s_n(x)$ is a polynomial of degree $n-1$. More precisely,
 \begin{equation}
\label{eq:lemma_Lns}
L^s_n(x) =  \Gamma(s) \sum_{k=0}^{n-1} \frac{(2s)_k}{k!} \frac{\Gamma( n - k - s+1)} { (s+ k - n ) \Gamma(n - k)} x^k.
\end{equation}
\end{lemma}
\begin{proof}
  We proceed by substituting $(1-y)^{s-1}$ in the
  integrand~\eqref{eq:Ks_n} by its Taylor expansion around $y=0$,
\begin{equation}\label{taylor_one_bnd}
  (1-y)^{s-1} = \sum_{j=0}^{\infty} q_j y^j, \mbox{ with } q_j = \frac{(1-s)_j}{j!}, 
\end{equation}
and subsequently exchanging the principal value integration with the
infinite sum (a step that is justified in
Appendix~\ref{appendix_pvseries}). The result is
\begin{equation}
\label{eq:series_Ks}
L^s_n(x) = \sum_{j=0}^{\infty}  \left( \mbox{P.V.} \int_0^1 \sgn(x-y)|x-y|^{-2s} q_j y^{s-1+n+j}  dy \right)
\end{equation}
or, in terms of the functions $N_\alpha^s$ defined in
equation~\eqref{eq:def_Ns},
\begin{equation}\label{lns2}
L^s_n(x) = \sum_{j=0}^{\infty} q_j N^s_{{s+n+j}} .
\end{equation}
%

In view of~\eqref{eq:teo1}, equation~\eqref{lns2} can also be made to
read
\begin{equation}
 \label{eq:series_ajk}
 L^s_n(x) = \sum_{j=0}^{\infty} \sum_{k=0}^{\infty} \frac{(1-s)_j}{j!} \frac{(2s)_k}{k!} \frac{1}{s-n-j+k} \, x^k,
\end{equation}
or, interchanging of the order of summation in this expression (which
is justified in Appendix~\ref{appendix_sumorder}),
\begin{equation}
\label{eq:Kz_series}
 L^s_n(x) =  \sum_{k=0}^{\infty} \frac{(2s)_k}{k!} a_{k}^n x^k, \mbox{ where }  a_{k}^n = \sum_{j=0}^{\infty} \frac{(1-s)_j}{j!} \frac{1}{s-n-j+k}.
\end{equation}
The proof will be completed by evaluating explicitly the coefficients
$a_k^n$ for all pairs of integers $k$ and $n$.

In order to evaluate $a_k^n$ we consider the Hypergeometric function
\begin{equation}\label{hyper_geom}
  _2F_1(a,b;c;z)=\sum_{j=0}^\infty \frac{(a)_j
    (b)_j}{(c)_j} \frac{z^j}{j!}.
\end{equation}
Comparing the $a_k^n$ expression in~\eqref{eq:Kz_series}
to~\eqref{hyper_geom} and taking into account the relation
$$\frac{1}{s-n-j+k} = \frac{(n-k-s)_j}{(n-k-s+1)_j} \, \frac{1}{s+k-n}$$
(which follows easily from the recursion $(z+1)_j= (z)_j (z+j)/z$ for
the Pochhamer symbol defined in equation~\eqref{def_Pochhamer}),
we see that $a_k^n$ can be expressed in terms of the Hypergeometric
function $_2F_1$ evaluated at $z=1$:
$$a_k^n = 2F_1(1-s,n-k-s;n-k-s+1;1)/(s+k-n). $$ 
This expression can be simplified further: in view of Gauss' formula
$_2F_1(a,b;c;1) =
\frac{\Gamma(c)\Gamma(c-a-b)}{\Gamma(c-a)\Gamma(c-b)}$ (see
e.g.~\cite[p. 2]{Bailey}) we obtain the concise expression
\begin{equation}
\label{eq:ak_finitos}
 a_{k}^n = \frac{\Gamma(n-k-s+1)\Gamma(s)} { (s+k-n) \Gamma(n - k) } .
\end{equation}
It then clearly follows that $a_{k}^n = 0$ for $k \ge n$---since the
term $\Gamma(n - k)$ in the denominator of this expression is infinite
for all integers $k\geq n$. The series in~\eqref{eq:Kz_series} is
therefore a finite sum up to $k=n-1$ which, in view
of~\eqref{eq:ak_finitos}, coincides with the desired
expression~\eqref{eq:lemma_Lns}. The proof is now complete.
\end{proof}
\begin{corollary}\label{cor_poly_w}
  Let $w(y) = u(y)\chi_{(0,1)}(y)$ where $u= y^s(1-y)^s y^n$
  (equation~\eqref{udef}) and where $\chi_{(0,1)}$ denotes the
  characteristic function of the interval $(0,1)$. Then, defining the
  $n$-th degree polynomial $p(x) = (1-2s)C_s \left( (s+n) L^s_n -
    (2s+n) L^s_{n+1} \right)$ with $L^s_n$ given
  by~\eqref{eq:lemma_Lns}, for all $x\in \mathbb{R}$ such that $x\ne
  0$ and $x\ne 1$ (cf. Remark~\ref{bound_rem}) we have
  \begin{equation}\label{Ts_pol}
T_s[u] (x) = p(x)
\end{equation}
and, consequently,
\begin{equation}\label{frac_pol}
(-\Delta)^s w(x) = p(x).
\end{equation}
\end{corollary}
\begin{proof}
  In view of equation~\eqref{eq:Ts_useful} and Lemma~\ref{lemma_Lns}
  we obtain~\eqref{Ts_pol}. The relation~\eqref{frac_pol} then follows
  from Remark~\ref{remark_Ts}.
\end{proof}

In view of equation~\ref{eq:Ss_useful} and Lemma~\ref{lemma_Lns}, the
results obtained for the image of $u(y) = y^{s}(1-y)^{s} y^n$ under
the operator $T_s$ can be easily adapted to obtain analogous
polynomial expressions of degree exactly $n$ for the image of the
function $\tilde{u}(y) = y^{s-1}(1-y)^{s-1} y^n$ under the operator
$S_{s}$.  And, indeed, both of these results can be expressed in terms
of isomorphisms in the space $\mathbb{P}_n$ of polynomials of degree
less or equal than $n$, as indicated in the following corollary,
\begin{corollary}
\label{coro_diagonal}
Let $s\in (0,1)$, $m\in \mathbb{N}$, and consider the linear mappings
$P:\mathbb{P}_m \to \mathbb{P}_m$ and $Q:\mathbb{P}_m \to
\mathbb{P}_m$ defined by
\begin{equation}\label{eq_PQ}
 \begin{split}
   P : p & \to T_s[y^s(1-y)^s p(y)] \;\;\; \mbox{and} \\
   Q : p & \to S_{s}[y^{s-1}(1-y)^{s-1} p(y)].
 \end{split}
\end{equation}
Then the matrices $[P]$ and $[Q]$ of the linear mappings $P$ and $Q$
in the basis $\{ y^n:n=0,\dots,m\}$ are upper-triangular and their
diagonal entries are given by
\begin{equation*}
 \begin{split}
   P_{nn} = &  \frac{\Gamma(2s+n+1) }{n!}\;\;\; \mbox{and}   \\
   Q_{nn} = & -\frac{\Gamma(2s+n-1) }{n!},
 \end{split}
\end{equation*}
respectively.  In particular, for $s = \frac 12$ we have
\begin{equation}\label{diagonal_entries}
 \begin{split}
  P_{nn} = & \;\; 2 n   \\
  Q_{nn} = & - \frac{2}{n}  \;\;\; \mbox{for } \;\;\; n\ne 0 \;\;\; \mbox{and }\;\;\;  Q_{00} = -2\log(2). \\
 \end{split} 
\end{equation}
\end{corollary}
\begin{proof}
  The expressions for $n\ne 0$ and for $P_{00}$ follow directly from
  equations \eqref{eq:Ts_useful}, \eqref{eq:Ss_useful}
  and~\eqref{eq:lemma_Lns}. In order to obtain $Q_{00}$, in turn, we
  note from~\eqref{eq:Ss_useful} that for $n=0$ we have
  $\frac{\partial}{\partial x} S_{s}( y^{s-1} (1-y)^{s-1} y^n )=0$.
  In particular, $S_{s}( y^{s-1} (1-y)^{s-1} )$ does not depend on $x$
  and we therefore obtain
\begin{equation*}
 \begin{split}
   Q_{00} =  S_{s}( y^{s-1} (1-y)^{s-1} ) &=  C_s \int_0^1 ( y^{2s-1} - 1 ) y^{s-1} (1-y)^{s-1} dy \\
   &= C_s \left( \mbox{B}(3s-1,s) - \mbox{B}(s,s) \right).
 \end{split}
\end{equation*} 
In the limit as $s \to 1/2$, employing l'H\^opital's rule together with
well known values\cite[6.1.8, 6.3.2, 6.3.3]{AbramowitzStegun} for the
Gamma function and it's derivative at $z=1/2$ and $z=1$, we obtain
$S_{\frac 12}( y^{-1/2} (1-y)^{-1/2} ) = -2\log(2)$
\end{proof}

\subsection{Diagonal Form of the Weighted Fractional Laplacian\label{diag}}
In view of the form of the mapping $P$ in equation~\eqref{eq_PQ} and
using the ``weight function''
$$\w^s(y) = (y-a)^s(b-y)^s,$$
for $\phi \in C^2(a,b)\cap C^1[a,b]$ (that is, $\phi$ smooth up to the
boundary but it does not necessarily vanish on the boundary) we
introduce the weighted version
\begin{equation}\label{eq:weighted_hypersingular}
K_s(\phi) = C_s \frac{d}{dx} \int_{a}^{b} |x-y|^{1-2s} \frac{d}{dy} \left( \w^s \phi(y) \right) dy \quad (s\ne 1/2),
\end{equation}
of the operator $T_s$ in equation~\eqref{Tsdef_eq3}. In view of Lemma~\ref{lemma_hypersingular}, $K_s$ can also be viewed as a weighted version of the Fractional Laplacian operator, and we  therefore define
\begin{equation}\label{eq:weighted_fractional}
(-\Delta)_\w^s[ \phi] = K_s(\phi) \;\; \mbox{for} \;\; \phi \in C^2(a,b)\cap C^1[a,b].
\end{equation}

\begin{remark}\label{rem_connection_u}
  Clearly, given a solution $\phi$ of the equation 
  \begin{equation}\label{eqn_weighted}
    (-\Delta)_\w^s[\phi] = f
\end{equation}
in the domain $\Omega = (a,b)$, the function $u= \w^s \phi$ extended
by zero outside $(a,b)$ solves the Dirichlet problem for the
Fractional Laplacian~\eqref{eq:fraccionario_dirichlet}
(cf. Lemma~\ref{lemma_hypersingular}).
\end{remark}

In order to study the spectral properties of the operator
$(-\Delta)^s_\w,$ consider the weighted $L^2$ space
\begin{equation}\label{weighted_L2}
  L^2_s(a,b) = \left\{ \phi:(a,b) \to \R \ \colon \int_a^b |\phi|^2 \w^s < \infty \right\},
\end{equation}
which, together with the inner product
\begin{equation}\label{scalarproduct_L2}
 (\phi,\psi)^s_{a,b} = \int_a^b \phi \, \psi \, \w^s 
\end{equation}
and associated norm is a Hilbert space. We can now establish the
following lemma.
\begin{lemma}
\label{teo:autoadj}
The operator $(-\Delta)_\w^s$  maps $\mathbb{P}_n$ into itself.
The restriction of $(-\Delta)_\w^s$ to $\mathbb{P}_n$ is a self
adjoint operator with respect to the inner product
$(\cdot,\cdot)^s_{a,b}$.
\end{lemma}
\begin{proof}
  Using the notation $K_s=(-\Delta)_\w^s$, we first establish the
  relation $(K_s[p],q) = (p,K_s[q])$ for $p,q \in \mathbb{P}_n$. But
  this follows directly from application of integration by parts and
  Fubini's theorem followed by an additional instance of integration
  by parts in \eqref{eq:weighted_hypersingular}, and noting that the
  the boundary terms vanish by virtue of the weight $\w^s$.
\end{proof}
The orthogonal polynomials with respect to the inner product under
consideration are the well known Gegenbauer
polynomials~\cite{AbramowitzStegun}. These are defined on the interval $(-1,1)$ by
the recurrence
\begin{equation}\label{eq:recurrencia}
\begin{split}
  C_{0}^{(\alpha)}(x) & = 1, \\
  C_{1}^{(\alpha)}(x) & = 2 \alpha x, \\
  C_{n}^{(\alpha)}(x) & = \frac{1}{n}
  \left[2x(n+\alpha-1)C_{n-1}^{(\alpha)}(x) -
    (n+2\alpha-2)C_{n-2}^{(\alpha)}(x) \right];
\end{split}\end{equation}
for an arbitrary interval $(a,b)$, the corresponding orthogonal
polynomials can be easily obtained by means of a suitable affine
change of variables. Using this orthogonal basis we can now produce an
explicit diagonalization of the operator $(-\Delta)^s_\w$.  We first consider the
interval $(0,1)$; the corresponding result for a general interval
$(a,b)$ is presented in Corollary~\ref{diag_ab}.
\begin{theorem} \label{teo:diagonalform} Given $s\in(0,1)$ and $n \in
  \mathbb{N}\cup \{ 0 \} $, consider the Gegenbauer polynomial
  $C^{(s+1/2)}_n$, and let $p_n(x) = C^{(s+1/2)}_n(2x-1)$. Then the
  weighted operator $(-\Delta)^s_\w$ in the interval $(0,1)$ satisfies
  the identity 
\begin{equation} \label{eq_diagform}
(-\Delta)^s_\w( p_n ) = 
\frac{\Gamma(2s+n+1)}{n!} \, p_n .
\end{equation}
\end{theorem}
\begin{proof}
  By Lemma~\ref{teo:autoadj} the restriction of the operator
  $(-\Delta)^s_\w$ to the subspace $\mathbb{P}_m$ is self-adjoint and
  thus diagonalizable. We may therefore select polynomials $q_0,
  q_1,\dots,q_m\in\mathbb{P}_m$ (where, for $0\leq n\leq m$, $q_n$ is
  a polynomial eigenfunction of $(-\Delta)^s_\w$ of degree exactly
  $n$) which form an orthogonal basis of the space
  $\mathbb{P}_m$. Clearly, the eigenfunctions $q_n$ are orthogonal and,
  therefore, up to constant factors, the polynomials $q_n$ must
  coincide with $p_n$ for all $n$, $0\leq n\leq m$.  The corresponding
  eigenvalues can be extracted from the diagonal elements, displayed
  in equation~\eqref{diagonal_entries}, of the upper-triangular matrix $[P]$
  considered in Corollary~\ref{coro_diagonal}. These entries coincide
  with the constant term in~\eqref{eq_diagform}, and the proof is
  thus complete.
\end{proof}

\begin{corollary}\label{diag_ab}
  The weighted operator $(-\Delta)^s_\w$ in the interval $(-1,1)$
  satisfies the identity
\[
(-\Delta)^s_\w(C_n^{(s+1/2)}) =   \lambda_n^s \, C_n^{(s+1/2)} ,
\]	
where
\begin{equation}
\label{Eigenvalues}
\lambda_n^s =  
\frac{\Gamma(2s+n+1)}{n!}.
\end{equation}
Moreover in the interval $(a,b)$, we have	
\begin{equation}\label{eigenfuncts} (-\Delta)^s_\w(p_n) =   
\lambda_n^s \, p_n ,
\end{equation}
where $p_n(x) = C_n^{(s+1/2)}\left(\frac{2(x-a)}{b-a} - 1 \right)$.  
\end{corollary}
\begin{proof}
  The formula is obtained by employing the change of variables
  $\tilde x=(x-a)/(b-a) $ and $\tilde y =(y-a)/(b-a)$ in equation~\eqref{eq:weighted_hypersingular}
  to map the weighted operator in $(a,b)$ to the corresponding operator in $(0,1)$, 
	and observing that $\w^s(y) = (b-a)^{2s} \tilde \w^s (\tilde y)$, where 
	$\tilde \w^s (\tilde y) = \tilde{y}^s (1 - \tilde y)^s.$
\end{proof}

\begin{remark}\label{eigenvalue_asymptotics}
  It is useful to note that, in view of the formula $\lim_{n\to\infty}
  n^{\beta-\alpha}\Gamma(n+\alpha)/\Gamma(n+\beta) = 1$ (see
  e.g.~\cite[6.1.46]{AbramowitzStegun}) we have the asymptotic
  relation $\lambda_n^s \approx O\left( n^{2s} \right)$ for the
  eigenvalues~\eqref{Eigenvalues}. This fact will be exploited in the
  following sections in order to obtain sharp Sobolev regularity
  results as well as regularity results in spaces of analytic
  functions.
\end{remark}

As indicated in the following corollary, the background developed in
the present section can additionally be used to obtain the
diagonal form of the operator $S_s$ for all $s\in (0,1)$. This
corollary generalizes a corresponding existing result
for the case $s=1/2$---for which, as indicated in
Remark~\ref{remark_openarcs}, the operator $S_{s}$ coincides with the
single-layer potential for the solution of the two-dimensional Laplace
equation outside a straight arc or ``crack''.
\begin{corollary} The weighted operator $\phi \to
  S_s[\w^{s-1} \phi]$ can be diagonalized in terms of the Gegenbauer
  polynomials $C_n^{(s-1/2)}$
\begin{equation*}
 S_{s} \left[\w^{s-1} C_n^{(s-1/2)} \right] = \mu_n^s C^{(s-1/2)}_n ,
\end{equation*}
where in this case the eigenvalues are given by
$$ \mu_n^s = 
- \frac{\Gamma(2s+n-1) } {n!} .$$
\end{corollary}
\begin{proof}
The proof for the interval $[0,1]$ is analogous to that of Theorem \ref{teo:diagonalform}. In this case, the eigenvalues are extracted from the diagonal entries of the upper triangular matrix $[Q]$ in equation \eqref{diagonal_entries}. A linear change of variables allows to obtain the desired formula for an arbitrary interval.
\end{proof}

\begin{corollary}
  In the particular case $s=1/2$ on the interval $(-1,1)$, the
  previous results amount, on one hand, to the known
  result~\cite[eq. 9.27]{Handscomb} (cf also~\cite{YanSloan}),
\begin{equation*}
\int_{-1}^1 \log|x-y| T_n(y) (1-y^2)^{-1/2} dy = 
\left\lbrace
  \begin{array}{rl}
        -\frac{\pi}{n} T_n  & \mbox{ for } n \ne 0 \\
        -2\log(2) & \mbox{ for } n = 0 \\
      \end{array}
    \right. 
\end{equation*}
(where $T_n$ denotes the Tchevyshev polynomial of the first kind),
and, on the other hand, to the relation
\begin{equation*}
\frac{\partial}{\partial x} \int_{-1}^1 \log|x-y| \frac{\partial}{\partial y} \left( U_n(y)  (1-y^2)^{1/2} \right) dy  =  (n+1) \pi U_n 
\end{equation*}
(where $U_n$ denotes the Tchevyshev polynomial of the second kind).
\end{corollary}

\section{Regularity Theory}
\label{regularity}
This section studies the regularity of solutions of the fractional
Laplacian equation~\eqref{eq:fraccionario_dirichlet} under various
smoothness assumptions on the right-hand side $f$--including
treatments in both Sobolev and analytic function spaces, and for
multi-interval domains $\Omega$ as in
Definition~\ref{union_intervals_def}.  In particular,
Section~\ref{Sobolev} introduces certain weighted Sobolev spaces
$H^r_{s}(\W)$ (which are defined by means of expansions in Gegenbauer
polynomials together with an associated norm). The space $A_\rho$ of
analytic functions in a certain ``Bernstein Ellipse''
$\mathcal{B}_\rho$ is then considered in
Section~\ref{single_interval_analytic}. The main result in
Section~\ref{Sobolev} (resp. Section~\ref{single_interval_analytic})
establishes that for right-hand sides $f$ in the space $H^r_{s}(\W)$
with $r\geq 0$ (resp. the space $A_\rho(\W)$ with $\rho >0$) the
solution $u$ of equation~\eqref{eq:fraccionario_dirichlet} can be
expressed in the form $u(x) = \w^s(x) \phi(x)$, where $\phi$ belongs
to $H^{r+2s}_s(\W)$ (resp. to $A_\rho(\W)$). Sections~\ref{Sobolev}
and~\ref{single_interval_analytic} consider the single-interval case;
generalizations of all results to the multi-interval context are
presented in Section~\ref{regularity_multi_int}. The theoretical
background developed in the present Section~\ref{regularity} is
exploited in Section~\ref{HONM} to develop and analyze a class of
effective algorithms for the numerical solution of
equation~\eqref{eq:fraccionario_dirichlet} in multi-interval domains
$\Omega$.

\subsection{Sobolev Regularity, single interval case}\label{Sobolev}

In this section we define certain weighted Sobolev spaces, which
provide a sharp regularity result for the weighted Fractional
Laplacian $(-\Delta)^s_\w$ (Theorem \ref{teo_extended}) as well as a
natural framework for the analysis of the high order numerical methods
proposed in Section~\ref{HONM}.  It is noted that these spaces
coincide with the non-uniformly weighted Sobolev spaces introduced
in~\cite{BabuskaGuo}; Theorem~\ref{gegen_embedding_Hr} below provides
an embedding of these spaces into spaces of continuously
differentiable functions.  For notational convenience, in the present
discussion leading to the definition~\ref{def:sobolev} of the Sobolev
space $H^r_s(\W)$, we restrict our attention to the domain $\W =
(-1,1)$; the corresponding definition for general multi-interval
domains then follows easily.

In order to introduce the weighted Sobolev spaces we note that the set
of Gegenbauer polynomials $C^{(s + 1/2)}_n$ constitutes an orthogonal
basis of $L^2_s(-1,1)$ (cf. ~\eqref{weighted_L2}). The $L^2_s$ norm of 
a Gegenbauer polynomial (see ~\cite[eq 22.2.3]{AbramowitzStegun}), is given by 
\begin{equation} 
\label{eq:norm_gegen}
h_j^{(s+1/2)} = \left\| C^{(s + 1/2)}_j \right\|_{L^2_s(-1,1)} = \sqrt{\frac{2^{-2s}\pi}{\Gamma^2(s + 1/2)} \frac{\Gamma(j+2s+1)}{\Gamma(j+1)(j+s+1/2)}}.
\end{equation}
\begin{definition}\label{normal}
  Throughout this paper $\tilde{C}^{(s + 1/2)}_j$ denotes the
  normalized polynomial $C^{(s + 1/2)}_j / h_j^{(s+1/2)}$.
\end{definition}
Given a function $v\in L^2_s(-1,1)$, we have the following expansion
\begin{equation}\label{exp_gegen}
v (x) = \sum_{j=0}^\infty v_{j,s} \tilde{C}^{(s + 1/2)}_j (x), 
\end{equation}
which converges in $L^2_s(-1,1)$, and where
\begin{equation}\label{gegen_coef}
  v_{j,s} = \int_{-1}^1 v(x) \tilde{C}^{(s + 1/2)}_j (x) (1-x^2)^s dx.
\end{equation}

In view of the expression
  \begin{equation}\label{gegen_poly_der}
  \frac{d}{dx} C^{(\alpha)}_j (x) = 2\alpha C_{j-1}^{(\alpha + 1)}(x), \ j \geq 1,
  \end{equation}
  for the derivative of a Gegenbauer polynomial (see e.g.~\cite[eq. 4.7.14]{szego}), we have
\begin{equation}
  \frac{d}{dx} \tilde{C}^{(s+1/2)}_j (x) = (2s+1) \frac{h^{(s+3/2)}_{j-1}}{h^{(s+1/2)}_j} \tilde{C}^{s+3/2}_{j-1}.
\end{equation}
Thus, using term-wise differentiation in~\eqref{exp_gegen} we may
conjecture that, for sufficiently smooth functions $v$, we have
\begin{equation}\label{k-der}
  v^{(k)}(x) =\sum_{j=k}^\infty v_{j-k,s+k}^{(k)} \tilde {C}_{j-k}^{(s+k+1/2)}(x)
\end{equation}
where $ v^{(k)}(x)$ denotes the $k$-th derivative of the function
$v(x)$ and where, calling
\begin{equation}\label{A_nk_def}
  A_j^{k} = \prod_{r=0}^{k-1} \frac{h^{(s+3/2+r)}_{j-1-r}}{h^{(s+1/2+r)}_{j-r}} (2(s+r)+1) = 2^k \frac{h^{(s+1/2+k)}_{j-k}}{h^{(s+1/2)}_{j}} \frac{\Gamma(s+1/2+k)}{\Gamma(s+1/2)},
\end{equation}
the coefficients in~\eqref{k-der} are given by
\begin{equation}\label{k-der-coeffs}
  v_{j-k,s+k}^{(k)} = A_j^{k}\, v_{j,s}.
\end{equation}

Lemma~\ref{lemma_gegen_parts} below provides, in particular, a
rigorous proof of~\eqref{k-der} under minimal hypothesis.  Further,
the integration by parts formula established in that lemma together
with the asymptotic estimates on the factors $B_j^{k}$ provided in
Lemma~\ref{A_jk_lemma}, then allow us to relate the smoothness of a
function $v$ and the decay of its Gegenbauer coefficients; see
Corolary~\ref{gegen_decay_coro}.
\begin{lemma}[Integration by parts]\label{lemma_gegen_parts}
  Let $k\in \mathbb{N}$ and let $v \in C^{k-2}[-1,1]$ such that for a
  certain decomposition $[-1,1] = \bigcup_{i=1}^n
  [\alpha_i,\alpha_{i+1}]$ ($-1=\alpha_1<\alpha_i<\alpha_{i+1}
  <\alpha_n=1$) and for certain functions $\tilde v_i \in
  C^{k}[\alpha_i,\alpha_{i+1}]$ we have $v(x)=\tilde v_i(x)$ for all
  $x\in (\alpha_i,\alpha_{i+1})$ and $1\leq i\leq n$.  Then for $j\geq
  k$ the $s$-weighted Gegenbauer coefficients $v_{j,s}$ defined in
  equation~\eqref{gegen_coef} satisfy
\begin{equation}\label{gegen_coef_parts_k}
\begin{split}
v_{j,s}  = B_j^{k} \int_{-1}^1 & \tilde v^{(k)}(x) \tilde{C}_{j-k}^{(s+k+1/2)}(x) (1-x^2)^{s+k} dx \\
 &-B_j^{k} \sum_{i=1}^n \left[ \tilde v^{(k-1)}(x) \tilde{C}_{j-k}^{(s+k+1/2)}(x) (1-x^2)^{s+k} \right]_{\alpha_i}^{\alpha_{i+1}},
\end{split}
\end{equation}
where
\begin{equation}\label{B_jk_def}
  B_j^k = \frac{h_{j-k}^{(s+k+1/2)}}{h_j^{(s+1/2)}}  \prod_{r=0}^{k-1} \frac{(2(s+r)+1)}{(j-r)(2s+r+j+1)}.
\end{equation}
With reference to equation~\eqref{A_nk_def}, further, we have
$A_j^k=\frac{1}{B_j^k}$. In particular, under the additional
assumption that $v \in C^{k-1}[-1,1]$ the
relation~\eqref{k-der-coeffs} holds.
\end{lemma}
\begin{proof}
  Equation~\eqref{gegen_coef_parts_k} results from iterated
  applications of integration by parts together with the
  relation~\cite[eq. 22.13.2]{AbramowitzStegun}
$$ \frac{\ell(2t+\ell+1)}{2t+1} \int (1-y^2)^{t} C_\ell^{(t+1/2)}(y) dy =  
- (1-x^2)^{t+1} C_{\ell-1}^{(t+3/2)}(x). $$ and subsequent
normalization according to Definition~\ref{normal}. The validity of
the relation $A_j^k=\frac{1}{B_j^k}$ can be checked easily.
\end{proof}

\begin{lemma}\label{A_jk_lemma} There exist constants $C_1$ and $C_2$ such that the
  factors $B_j^k$ in equation~\eqref{A_nk_def} satisfy
\begin{equation*}\label{A_jk_inequality}
  C_1 j^{-k} < |B_j^k| < C_2 j^{-k}
\end{equation*}
\end{lemma} 
\begin{proof}
  In view of the relation $\lim_{j\to\infty}
  j^{b-a}\Gamma(j+a)/\Gamma(j+b) = 1$
  (see~\cite[6.1.46]{AbramowitzStegun}) it follows that
  $h_j^{(s+1/2)}$ in equation~\eqref{eq:norm_gegen} satisfies
  \begin{equation}\label{gegen_norm_est}
   \lim_{j\to\infty} j^{1/2-s} h_j^{(s+1/2)} \ne 0
  \end{equation}
  and, thus, letting
\begin{equation}\label{gegen_ratio}
  q_j^k=\frac{h_{j-k}^{(s+k+1/2)}}{h_j^{(s+1/2)}},
\end{equation}
we obtain 
\begin{equation}\label{q_bound}
  \lim_{j\to\infty} q_j^k/j^{k} \ne 0.
\end{equation}
The lemma now follows by estimating the asymptotics of the product
term on the right-hand side of~\eqref{B_jk_def} as $j\to\infty$.
\end{proof}
\begin{corollary}\label{gegen_decay_coro}
  Let $k\in \mathbb{N}$ and let $v$ satisfy the hypothesis of
  Lemma~\ref{lemma_gegen_parts}. Then the Gegenbauer coefficients
  $v_{j,s}$ in equation~\eqref{gegen_coef} are quantities of order
  $O(j^{-k})$ as $j\to\infty$:
$$|v_{j,s}| < C j^{-k}$$
for a constant $C$ that depends on $v$ and $k$.
\end{corollary}
\begin{proof}
  The proof of the corollary proceeds by noting that the factor
  $B_j^{k}$ in equation~\eqref{gegen_coef_parts_k} is a quantity of
  order $j^{-k}$ (Lemma~\ref{A_jk_lemma}), and obtaining bounds for
  the remaining factors in that equation.  These bounds can be
  produced by (i)~applying the Cauchy-Schwartz inequality in the space
  $L^2_{s+k}(-1,1)$ to the $(s+k)$-weighted scalar
  product~\eqref{scalarproduct_L2} that occurs in
  equation~\eqref{gegen_coef_parts_k}; and
  (ii)~using~\cite[eq. 7.33.6]{szego} to estimate the boundary terms
  in equation~\eqref{gegen_coef_parts_k}. The derivation of the bound
  per point~(i) is straightforward. From~\cite[eq. 7.33.6]{szego}, on
  the other hand, it follows directly that for each $\lambda >0$ there
  is a constant $C$ such that
$$ |\sin(\theta)^{2\lambda-1} C^{\lambda}_j(\cos(\theta))| \le C j^{\lambda-1}. $$
Letting $x=\cos(\theta)$, $\lambda =s+k+1/2$ and dividing by the
normalization constant $h_{j}^{(s+k+1/2)}$ we then obtain
$$ \left |\tilde{C}^{s+k+1/2}_j(x)(1-x^2)^{s+k}\right | < Cj^{s+k-1/2} / h_{j}^{(s+k+1/2)}.$$ 
In view of~\eqref{gegen_norm_est}, the right hand side in this equation
is bounded for all $j\geq 0$. The proof now follows from
Lemma~\ref{A_jk_lemma}.
\end{proof}


We now define a class of Sobolev spaces $H^r_s$ 
that, as shown in Theorem~\ref{teo_extended}, completely characterizes
the Sobolev regularity of the weighted fractional Laplacian operator
$(-\Delta)^s_\w$. 
\begin{remark}\label{no-s}
  In what follows, and when clear from the context, we drop the
  subindex $s$ in the notation for Gegenbauer coefficients such as
  $v_{j,s}$ in~\eqref{gegen_coef}, and we write e.g. $v_j=v_{j,s}$,
  $w_j=w_{j,s}$, $f_j=f_{j,s}$, etc.
\end{remark}

\begin{definition}\label{def:sobolev} Let $r,s\in\mathbb{R}$,  
  $r\geq 0$, $s> -1/2$ and, for $v \in L^2_s(-1,1)$ call $v_j$ the
  corresponding Gegenbauer coefficient~\eqref{gegen_coef} (see
  Remark~\ref{no-s}). Then the complex vector space $H^r_s(-1,1) =
  \left\{ v \in L^2_s(-1,1) \colon \sum_{j=0}^\infty (1+j^{2})^r
    |v_j|^2 < \infty \right\}$ will be called the $s$-weighted Sobolev
  space of order $r$.
\end{definition}
\begin{lemma}\label{lemma_hilbert_space}
  Let $r,s\in\mathbb{R}$, $r\geq 0$, $s> -1/2$. Then the
  space $H^r_s(-1,1)$ endowed with the inner product $\langle v, w
  \rangle_s^r = \sum_{j=0}^\infty v_j w_j (1 + j^{2})^r$ and
  associated norm 
\begin{equation}\label{H_r_norm}
  \| v \|_{H_s^r} = \sum_{j=0}^\infty |v_j|^2 (1 + j^{2}
  )^r
\end{equation}
 is a Hilbert space.
\end{lemma}
\begin{proof}
  The proof is completely analogous to that of \cite[Theorem
  8.2]{Kress}.
\end{proof}
\begin{remark}\label{gegen_aprox_Hs} 
  By definition it is immediately checked that for every function
  $v\in H^r_s(-1,1)$ the Gegenbauer expansion~\eqref{exp_gegen} with
  expansion coefficients~\eqref{gegen_coef} is convergent in
  $H_s^r(-1,1)$.
\end{remark}
\begin{remark}\label{sobolev_remark}
  In view of the Parseval identity $\|v\|_{L^2_s(-1,1)}^2 =
  \sum_{n=0}^\infty |v_n|^2$ it follows that the Hilbert spaces
  $H^0_s(-1,1)$ and $L^2_s(-1,1)$ coincide. Further, we have the dense
  compact embedding $H^t_s(-1,1)\subset H^r_s(-1,1)$ whenever $r<t$.
  (The density of the embedding follows directly from
  Remark~\ref{gegen_aprox_Hs} since all polynomials are contained in
  $H^r_s(-1,1)$ for every $r$.) Finally, by proceeding as in
  \cite[Theorem 8.13]{Kress} it follows that for any $r>0$,
  $H^r_s(-1,1)$ constitutes an interpolation space between $H^{\lfloor
    r \rfloor}_s(-1,1)$ and $H^{\lceil r \rceil}_s(-1,1)$ in the sense
  defined by \cite[Chapter 2]{BerghLofstrom}.
\end{remark}

Closely related ``Jacobi-weighted Sobolev spaces'' $\mathcal{H}^k_s$
(Definition~\ref{Guo_def}) were introduced
previously~\cite{BabuskaGuo} in connection with Jacobi approximation
problems in the $p$-version of the finite element method. As shown in
Lemma~\ref{sobolev_equivalence} below, in fact, the spaces
$\mathcal{H}^k_s$ coincide with the spaces $H^k_s$ defined above, and
the respective norms are equivalent.
\begin{definition}[Babu\v{s}ka and Guo~\cite{BabuskaGuo}]\label{Guo_def}
  Let $k\in\mathbb{N}\cup \{0\}$ and $r>0$. The $k$-th order
  non-uniformly weighted Sobolev space $\mathcal{H}^k_s(a,b)$ is
  defined as the completion of the set $C^\infty(a,b)$ under the norm
\begin{equation*}\label{gou_norm}
  \|v \|_{\mathcal{H}^k_s} = \left( \sum_{j=0}^k \int_{a}^b |v^{(j)}(x)|^2 \w^{s+j} dx \right)^{1/2} = \left( \sum_{j=0}^k \| v^{(j)} \|_{L^2_{s+j}}^2 \right)^{1/2}.
\end{equation*}  
The $r$-th order space $\mathcal{H}^r_s(a,b)$, in turn, is defined by
interpolation of the spaces $\mathcal{H}^k_s(a,b)$
($k\in\mathbb{N}\cup \{0\}$) by the $K$-method (see ~\cite[Section
3.1]{BerghLofstrom}).
\end{definition}

\begin{lemma}\label{sobolev_equivalence}
  Let $r>0$. The spaces $H^r_s(a,b)$ and $\mathcal{H}^r_s(a,b)$
  coincide, and their corresponding norms $\| \cdot \|_{H^r_s}$ and $\|
  \cdot \|_{\mathcal{H}^r_s}$ are equivalent.
\end{lemma}
\begin{proof}
  A proof of this lemma for all $r>0$ can be found in \cite[Theorem
  2.1 and Remark 2.3]{BabuskaGuo}. In what follows we present an
  alternative proof for non-negative integer values of $r$:
  $r=k\in\mathbb{N}\cup \{0\}$. In this case it suffices to show that
  the norms $\| \cdot \|_{H^k_s}$ and $\| \cdot \|_{\mathcal{H}^k_s}$
  are equivalent on the dense subset $C^\infty[a,b]$ of both
  $H^k_s(a,b)$ (Remark~\ref{gegen_aprox_Hs}) and
  $\mathcal{H}^k_s(a,b)$. But, for $v\in C^\infty[a,b]$,
  using~\eqref{k-der}, Parseval's identity in $L^2_{s+k}$ and
  Lemma~\ref{lemma_gegen_parts} we see that for every integer $k\geq
  0$ we have $\| v^{(k)} \|_{L^2_{s+k}} = \sum_{j=k}^\infty
  |v_{j-k,s+k}^{(k)}|^2 = \sum_{j=k}^\infty |v_{j,s}|^2 /
  |B_j^k|^2$. From Lemma~\ref{A_jk_lemma} we then obtain
$$ D_1 \sum_{j=k}^\infty |v_{j,s}|^2 j^{2k} \le  \| v^{(k)} \|_{L^2_{s+k}}^2 \le  D_2 \sum_{j=k}^\infty |v_{j,s}|^2 j^{2k} $$
for certain constants $D_1$ and $D_2$, where $v_{j-k,s+k}^{(k)}$.  In
view of the inequalities
 $$ (1+j^{2k}) \le (1+j^2)^k \le (2j^2)^k \le 2^k(1+j^{2k}) $$ 
 the claimed norm equivalence for $r=k\in\mathbb{N}\cup \{0\}$ and
 $v\in C^\infty[a,b]$ follows. 
\end{proof}

Sharp regularity results for the Fractional Laplacian in the Sobolev space
$H^r_s(a,b)$ can now be obtained easily.
\begin{theorem}\label{teo_extended}
  Let $r\geq 0$. Then the weighted fractional Laplacian
  operator~\eqref{eq:weighted_fractional} can be extended uniquely to
  a continuous linear map $(-\Delta)^s_\w$ from $H_s^{r+2s}(a,b)$ into
  $H_s^{r}(a,b)$. The extended operator is bijective and bicontinuous.
\end{theorem}
\begin{proof} Without loss of generality, we assume $(a,b)=(-1,1)$.
  Let $\phi\in H^{r+2s}_s(-1,1)$, and let $\phi^n = \sum_{j=0}^n
  \phi_j \tilde{C}_j^{(s+1/2)} $ where $\phi_j$ denotes the Gegenbauer
  coefficient of $\phi$ as given by equation~\eqref{gegen_coef} with
  $v=\phi$.  According to Corollary~\ref{diag_ab}  we have
  $(-\Delta)^s_\w \phi^n = \sum_{j=0}^n \lambda_j^s \phi_j
  \tilde{C}_j^{(s+1/2)}$. In view of Remarks~\ref{gegen_aprox_Hs}
  and~\ref{eigenvalue_asymptotics} it is clear that $(-\Delta)^s_\w
  \phi^n$ is a Cauchy sequence (and thus a convergent sequence) in
  $H_s^{r}(-1,1)$. We may thus define 
$$(-\Delta)^s_\w \phi = \lim_{n\to\infty}
(-\Delta)^s_\w \phi^n = \sum_{j=0}^\infty\lambda_j^s \phi_j
\tilde{C}_j^{(s+1/2)}\in H_s^{r}(-1,1).$$ 
The bijectivity and bicontinuity of the
extended mapping follows easily, in view of Remark~\ref{eigenvalue_asymptotics}, 
as does the uniqueness of continuous extension. The proof is complete.
\end{proof}
\begin{corollary} \label{coro_u_sobolev} The solution $u$
  of~\eqref{eq:fraccionario_dirichlet} with right-hand side $f\in
  H^r_s(a,b)$ ($r\geq 0$) can be expressed in the form $u=\w^s\phi$ for
  some $\phi \in H^{r+2s}_s(a,b)$.
\end{corollary}
\begin{proof}
  Follows from Theorem~\ref{teo_extended} and
  Remark~\ref{rem_connection_u}.
\end{proof}

The classical smoothness of solutions of
equation~\eqref{eq:fraccionario_dirichlet} for sufficiently smooth
right-hand sides results from the following version of the ``Sobolev
embedding'' theorem.
\begin{theorem}[Sobolev's Lemma for weighted
  spaces]\label{gegen_embedding_Hr}
  Let $s \ge 0$, $k\in \mathbb{N}\cup \{0\}$ and $r > 2k + s +
  1$. Then we have a continuous embedding $H^r_s(a,b)\subset C^k[a,b]$
  of $H^r_s(a,b)$ into the Banach space $C^k[a,b]$ of $k$-continuously
  differentiable functions in $[a,b]$ with the usual norm $\| v \|_k$
  (given by the sum of the $L^\infty$ norms of the function and the
  $k$-th derivative): $\| v \|_k = \| v \|_\infty + \| v^{(k)}
  \|_\infty$.
\end{theorem}
\begin{proof} Without loss of generality we restrict attention to
  $(a,b)=(-1,1)$.  Let $0\le \ell \le k$ and let $v\in H^r_s(-1,1)$ be
  given. Using the expansion~\eqref{exp_gegen} and in view of the
  relation~\eqref{gegen_poly_der} for the derivative of a Gegenbauer polynomial, we
  consider the partial sums
  \begin{equation}\label{exp_gegen_der}
  v^{(\ell)}_n(x) = 2^\ell  \prod_{i=1}^\ell (s + i - 1/2)
  \sum_{j=\ell}^n \frac{v_j}{h_j^{(s+1/2)}} C^{(s+\ell+1/2)}_{j-\ell}(x)
  \end{equation}
  that result as the partial sums corresponding to~\eqref{exp_gegen} up to $j=n$ are differentiated $\ell$ times.
  But we have the estimate
  \begin{equation}\label{bound_gegen}
  \|C^{(s+1/2)}_{n}\|_\infty \sim O(n^{2s}).
  \end{equation}
  which is an immediate consequence of~\cite[Theorem 7.33.1]{szego}. Thus, taking into account~\eqref{gegen_norm_est}, we obtain 
  $$ |v_n^{(\ell)}(x)| \le C(\ell) \sum_{j=0}^{n-\ell} \frac{| v_{j+\ell}|} {h_{j+\ell}^{(s+1/2)}} |C^{(s+\ell+1/2)}_{j}(x)| 
  \le C(\ell) \sum_{j=0}^{n-\ell} (1+j^2)^{(s+2\ell)/2 + 1/4} | v_{j+\ell}| ,$$
for some constant $C(\ell)$. Multiplying and dividing by $(1+j^2)^{r/2}$ and
applying the Cauchy-Schwartz inequality in the space of square
summable sequences it follows that
\begin{equation}\label{g_der}
| v_n^{(\ell)}(x) | \le C(\ell) \left(\sum_{j=0}^{n-\ell} \frac{1}{(1+j^2)^{r - (s+2\ell+1/2) }}\right)^{1/2} \left(\sum_{j=0}^{n-\ell} (1+j^2)^r |v_{j+\ell}|^2\right)^{1/2}. 
\end{equation}
We thus see that, provided $r - (s + 2\ell+1/2)>1/2$ (or equivalently,
$r>2\ell+s+1$), $v_n^{(\ell)}$ converges uniformly to
$\frac{\partial^\ell}{\partial x^\ell} v(x)$ (cf. ~\cite[Th. 7.17]{baby_rudin}) for all 
$\ell$ with
$0\leq \ell \leq k$. It follows that $v\in C^k[-1,1]$, and, in view
of~\eqref{g_der}, it is easily checked that there exists a constant
$M(\ell)$ such that $\| \frac{\partial^{(\ell)}}{\partial x^k} v(x)
\|_\infty \le M(\ell) \| v \|_s^r$ for all $0\leq \ell \leq k.$ The
proof is complete.
\end{proof}

\begin{remark}
In order to check that the previous result is sharp we consider an example in the case $k=0$: 
the function $v(x)=|\log(x)|^{\beta}$ with $0<\beta<1/2$ is not bounded, but a straightforward computation 
shows that, for $s \in  \N$, $v\in \mathcal{H}_s^{s+1}(0,1)$, or equivalently (see Lemma \ref{sobolev_equivalence}),
$v\in H_s^{s+1}(0,1)$.
\end{remark}

\begin{corollary}\label{cinfty}
  The weighted fractional Laplacian
  operator~\eqref{eq:weighted_fractional} maps bijectively the
  space $C^\infty[a,b]$ into itself.
\end{corollary}
\begin{proof}
 Follows directly from Theorem~\ref{teo_extended} together 
 with lemmas~\ref{lemma_gegen_parts},~\ref{A_jk_lemma} and~\ref{gegen_embedding_Hr}.
\end{proof}
\subsection{Analytic Regularity, single interval case}
\label{single_interval_analytic}
Let $f$ denote an analytic function defined in the closed interval
$[-1,1]$. Our analytic regularity results for the solution of
equation~\eqref{eq:fraccionario_dirichlet} relies on consideration of
analytic extensions of the function $f$ to relevant neighborhoods of
the interval $[-1,1]$ in the complex plane. We thus consider the {\em
  Bernstein ellipse} $\mathcal{E}_\rho$, that is, the ellipse with
foci $\pm 1$ whose minor and major semiaxial lengths add up to
$\rho\geq 1$. We also consider the closed set $\mathcal{B}_\rho$ in
the complex plane which is bounded by $\mathcal{E}_\rho$ (and which
includes $\mathcal{E}_\rho$, of course). Clearly, any analytic
function $f$ over the interval $[-1,1]$ can be extended analytically
to $\mathcal{B}_\rho$ for some $\rho > 1$. We thus consider the
following set of analytic functions.
\begin{definition}\label{A_rho_def}
  For each $\rho >1$ let $A_\rho$ denote the normed space of analytic
  functions $A_\rho = \{ f \colon f \text{ is analytic on }
  \mathcal{B}_\rho\} $ endowed with the $L^\infty$ norm
  $\|\cdot\|_{L^\infty\left (\mathcal{B}_{\rho}\right)}$.
\end{definition}
\begin{theorem}
\label{teo_analyticity}
For each $f \in A_\rho$ we have $((-\Delta)^s_\w)^{-1} f \in
A_\rho$. Further, the mapping $((-\Delta)^s_\w)^{-1}: A_\rho \to
A_\rho$ is continuous.
\end{theorem}
\begin{proof}
  Let $f \in A_\rho$ and let us consider the Gegenbauer expansions
\begin{equation}\label{expansions}
  f=\sum_{j=0}^\infty f_j \tilde{C}_j^{(s+1/2)} \quad\mbox{and}\quad ((-\Delta)^s_\w)^{-1} f=\sum_{j=0}^\infty (\lambda_j^s)^{-1}  f_j \tilde{C}_j^{(s+1/2)}.
\end{equation}
In order to show that $((-\Delta)^s_\w)^{-1} f\in A_\rho$ it suffices
to show that the right-hand series in this equation converges
uniformly within $\mathcal{B}_{\rho_1}$ for some $\rho_1 > \rho$. To
do this we utilize bounds on both the Gegenbauer coefficients and the
Gegenbauer polynomials themselves.

In order to obtain suitable coefficient bounds, we note that, since $f
\in A_\rho$, there indeed exists $\rho_2 > \rho$ such that $f \in
A_{\rho_2}$. It follows~\cite{ZhaoWangXie} that the Gegenbauer
coefficients decay exponentially. More precisely, for a certain
constant $C$ we have the estimate
\begin{equation} \label{eq:cota_an} |f_j| \le C
  \max_{z\in\mathcal{B}_{\rho_2}} |f(z)| \rho_2^{-j}
  j^{-s}\quad\mbox{for some}\quad \rho_2>\rho,
\end{equation} 
which follows directly from corresponding bounds~\cite[eqns 2.28, 2.8,
1.1, 2.27]{ZhaoWangXie} on Jacobi coefficients.  (Here we have used
the relation
$$ C_j^{(s+1/2)} = r^s_j P_j^{(s,s)} \quad \mbox{with} \quad r^s_j = \frac{(2s+1)_j}{(s+1)_j} = O(j^{s}) $$
that expresses Gegenbauer polynomials $C_j^{(s+1/2)}$ in terms of
Jacobi polynomials $P_j^{(s,s)}$.)

In order to the adequately account for the growth of the Gegenbauer
polynomials, on the other hand, we consider the estimate
\begin{equation}
\label{eq:crecimiento_gegenbauer}
\frac{\| C_j^{(s+1/2)} \|_{L^\infty(\mathcal{B}_{\rho_1})}}{h_j^{(s+1/2)}} \le 
D \rho_1^j\quad\mbox{for all}\quad \rho_1>1,
\end{equation}
which follows directly from~\cite[Theorem 3.2]{XieWangZhao} and
equation~\eqref{gegen_norm_est}, where $D=D(\rho_1)$ is a constant
which depends on $\rho_1$.

Let now $\rho_1\in [\rho,\rho_2)$. In view of~\eqref{eq:cota_an}
and~\eqref{eq:crecimiento_gegenbauer} we see that the $j$-th term of
the right-hand series in equation~\eqref{expansions} satisfies
\begin{equation}\label{series_est}
\left | \frac{\lambda_j^s  f_j C_j^{(s+1/2)}(x)}{h_j^{(s+1/2)}}\right |\leq C D(\rho_1)
\left( \frac{\rho_1}{\rho_2} \right)^j j^{-s} (\lambda_j^s)^{-1}  \max_{z\in\mathcal{B}_{\rho_1}} |f(z)|
\end{equation} 
throughout $\mathcal{B}_{\rho_1}$.  Taking $\rho_1\in (\rho,\rho_2)$
we conclude that the series converges uniformly in
$\mathcal{B}_{\rho_1}$, and that the limit is therefore analytic
throughout $\mathcal{B}_{\rho}$, as desired. Finally, taking
$\rho_1=\rho$ in~\eqref{series_est} we obtain the estimates
\begin{equation*}
  \| ((-\Delta)^s_\w)^{-1} f \|_{L^\infty(\mathcal{B}_{\rho})} \le C D(\rho) \sum_{j=0}^\infty \left( \frac{\rho}{\rho_2} \right)^j j^{-s} (\lambda_j^s)^{-1}  \max_{z\in\mathcal{E}_{\rho}} |f(z)| 
  \le E \|f \|_{L^\infty(\mathcal{B}_{\rho})}
\end{equation*} 
which establish the stated continuity condition. The proof is thus
complete.
\end{proof}

\begin{corollary}
  Let $f \in A_\rho$. Then the solution $u$ of
  \eqref{eq:fraccionario_dirichlet} can be expressed in the form
  $u=\w^s\phi$ for a certain $\phi \in A_\rho$.
\end{corollary}
\begin{proof}
  Follows from Theorem~\ref{teo_analyticity} and
  Remark~\ref{rem_connection_u}.
\end{proof}

\begin{remark}\label{remark_2s}
  We can now see that, as indicated in Remark~\ref{remark_idea_2s},
  the smoothness and analyticity theory presented throughout
  Section~\ref{regularity} cannot be duplicated with weights of
  exponent $2s$, in spite of the ``local'' regularity result of
  Theorem~\ref{teo1}---that establishes analyticity of
  $T[y^{\alpha}](x)$ around $x=0$ for both cases, $\alpha = s+n$ and
  $\alpha = 2s+n$.  Indeed, we can easily verify that
  $T(y^{2s}(1-y)^{2s} y^n)$ cannot be extended analytically to an open
  set containing $[0,1]$. If it could, Theorem~\ref{teo_analyticity}
  would imply that $y^{s}(1-y)^{s}$ is an analytic function around
  $y=0$ and $y=1$. 
\end{remark}

\subsection{Sobolev and Analytic Regularity on Multi-interval
  Domains}\label{regularity_multi_int}
This section concerns multi-interval domains $\W$ of the
form~\eqref{union_intervals}. Using the characteristic functions
$\chi_{(a_i,b_i)}$ of the individual component interval, letting
$\w^s(x)=\sum_{i=1}^M(x-a_i)^s(b_i-x)^s \chi_{(a_i,b_i)}(x)$ and
relying on Corollary~\ref{coro_lemma_hypersingular}, we define the
multi-interval weighted fractional Laplacian operator on $\Omega$ by
$(-\Delta)^s_\w \phi = (-\Delta)^s[\w^s \phi]$, where $\phi: \mathbb{R} \to
\mathbb{R}$. 
In view of the various
results in previous sections it is natural to use the decomposition
$(-\Delta)^s_\w = \mathcal{K}_s + \mathcal{R}_s$, where $\mathcal{K}_s[\phi] =
\sum_{i=1}^M\chi_{(a_i,b_i)} K_s\chi_{(a_i,b_i)}\phi$ is a
block-diagonal operator and where $\mathcal{R}_s$ is the associated off-diagonal
remainder. Using integration by parts is easy to check that
\begin{equation}\label{other_intervals}
  \mathcal{R}_s\phi(x) = C_1(s) \int_{\W\setminus (a_j,b_j)} |x-y|^{-1-2s} \w^s(y) \phi(y) dy\quad\mbox{for}\quad  x\in (a_j,b_j).
\end{equation}

\begin{theorem}
  Let $\Omega$ be given as in
  Definition~\ref{union_intervals_def}. Then, given $f \in L^2_s(\W)$,
  there exists a unique $\phi \in L^2_s(\W)$ such that
  $ (-\Delta)^s_\w \phi = f$.  Moreover, for $f \in H^r_s(\W)$
  (resp. $f \in A_\rho(\W)$) we have $\phi \in H^{r+2s}_s(\W)$
  (resp. $\phi \in A_\nu(\W)$ for some $\nu >1$).
\end{theorem} 
\begin{proof}
  Since $(-\Delta)^s_\w=(\mathcal{K}_s+\mathcal{R}_s)$,
  left-multiplying the equation $ (-\Delta)^s_\w \phi = f$ by
  $\mathcal{K}_s^{-1}$ yields
\begin{equation}
\label{eq:Preconditioner}
\left( I + \mathcal{K}_s^{-1}\mathcal{R}_s \right) \phi = \mathcal{K}_s^{-1} f .
\end{equation}
The operator $\mathcal{K}_s^{-1}$ is clearly compact in $L^2_s(\W)$
since the eigenvalues $\lambda_j^s$ tend to infinity as $j\to\infty$
(cf. \eqref{Eigenvalues}).  On the other hand, the kernel of the
operator $\mathcal{R}_s$ is analytic, and therefore $\mathcal{R}_s$ is
continuous (and, indeed, also compact) in $L^2_s(\W)$.  It follows
that the operator $\mathcal{K}_s^{-1}\mathcal{R}_s$ is compact in
$L^2_s(\W)$, and, thus, the Fredholm alternative tells us that
equation~\eqref{eq:Preconditioner} is uniquely solvable in $L^2_s(\W)$
provided the left-hand side operator is injective.

To establish the injectivity of this operator, assume $\phi \in L^2_s$
solves the homogeneous problem. Then
$\mathcal{K}_s(\phi) = - \mathcal{R}_s(\phi)$, and since
$\mathcal{R}_s(\phi)$ is an analytic function of $x$, in view of the
mapping properties established in Theorem~\ref{teo_analyticity} for
the self operator $K_s$ (which coincides with the
single-interval version of the operator $(-\Delta)^s_\w)$), we
conclude the solution $\phi$ to this problem is again analytic. Thus,
a solution to~\eqref{eq:fraccionario_dirichlet} for a null right-hand
side $f$ can be expressed in the form be $u = \w^s\phi$ for a certain
function $\phi$ which is analytic throughout $\Omega$.  But this
implies that the function $u = \w^s\phi$ belongs to the classical
Sobolev space $H^s(\W)$.  (To check this fact we consider that
(a)~$\w^s\in H^s(\W)$, since, by definition, the Fourier transform of
$\w^s$ coincides (up to a constant factor) with the confluent
hypergeometric function $M(s+1,2s+2,\xi)$ whose
asymptotics~\cite[eq. 13.5.1]{AbramowitzStegun} show that $\w^s$ in
fact belongs to the classical Sobolev space $H^{s+1/2-\eps}(\W)$ for
all $\eps>0$; and (b)~the product $fg$ of functions $f$, $g$ in
$H^s(\W)\cap L^\infty(\W)$ is necessarily an element of $H^s(\W)$---as
the Aronszajn-Gagliardo-Slobodeckij semi-norm~\cite{Hitchhikers} of
$fg$ can easily be shown to be finite for such functions $f$ and $g$,
which implies $fg \in H^s(\W)$~\cite[Prop 3.4]{Hitchhikers}).  Having
established that $u = \w^s\phi\in H^s(\W)$, the injectivity of the
operator in~\eqref{eq:Preconditioner} in $L^2_s(\W)$ follows from the
uniqueness of $H^s$ solutions, which is established for example
in~\cite{AcostaBorthagaray}. As indicated above, this injectivity
result suffices to establish the claimed existence of an $L^2_s(\W)$
solution for each $L^2_s(\W)$ right-hand side.

Assuming $f$ is analytic (resp. belongs to $H^r_s(\W)$), finally, the
regularity claims now follow directly from the single-interval results
of Sections~\ref{Sobolev} and~\ref{single_interval_analytic}, since a
solution $\phi$ of $(-\Delta)^s_\w \phi = f$ satisfies
\begin{equation}\label{multi_single}
\mathcal{K}_s(\phi) = f - \mathcal{R}_s(\phi).
\end{equation}
The proof is now complete.
\end{proof}


\section{High Order Numerical Methods\label{HONM}}
This section presents rapidly-convergent numerical methods for single-
and multi-interval fractional Laplacian problems.  In particular, this
section establishes that the proposed methods, which are based on the
theoretical framework introduced above in this paper, converge
(i)~exponentially fast for analytic right-hand sides $f$,
(ii)~superalgebraically fast for smooth $f$, and (iii)~with
convergence order $r$ for $f \in H_s^r(\W)$.

\subsection{Single-Interval Method: Gegenbauer Expansions\label{num_single}}
In view of Corollary~\ref{diag_ab}, a spectrally accurate algorithm
for solution of the single-interval equation~\eqref{eqn_weighted} (and
thus equation~\eqref{eq:fraccionario_dirichlet} for $\Omega=(a,b)$)
can be obtained from use of Gauss-Jacobi quadratures. Assuming
$(a,b)=(-1,1)$ for notational simplicity, the method proceeds as
follows: 1)~The continuous scalar product~\eqref{gegen_coef} with
$v=f$ is approximated with spectral accuracy (and, in fact, exactly
whenever $f$ is a polynomial of degree less or equal to $n+1$) by
means of the discrete inner product
\begin{equation}
\label{eq:disc_inner}
f_j^{(n)}:= \frac{1}{h_j^{(s+1/2)}} \sum_{i=0}^n f(x_i) C^{(s+1/2)}_j(x_i)  w_i,  
\end{equation}
where $(x_i)_{i=0}^n$ and $(w_i)_{i=0}^n$ denote the nodes and weights
of the Gauss-Jacobi quadrature rule of order $2n+1$. (As is well
known~\cite{HaleTowsend}, these quadrature nodes and weights can be
computed with full accuracy at a cost of $O(n)$ operations.)  2)~For
each $i$, the necessary values $C^{(s+1/2)}_j(x_i)$ can be obtained
for all $j$ via the three-term recurrence
relation~\eqref{eq:recurrencia}, at an overall cost of $O(n^2)$
operations. The algorithm is then completed by 3)~Explicit evaluation
of the spectrally accurate approximation
\begin{equation}
\label{eq:geg_rec}
\phi_n := K_{s,n}^{-1} f = \sum_{j=0}^n \frac{f_j^{(n)}}{\lambda_j^s h_j^{(s+1/2)}} C^{(s+1/2)}_j
\end{equation}
that results by using the expansion~\eqref{exp_gegen} with $v=f$
followed by an application of equation~\eqref{eigenfuncts} and
subsequent truncation of the resulting series up to $j=n$. The
algorithm requires accurate evaluation of certain ratios of Gamma
functions of large arguments; see equations~\eqref{Eigenvalues}
and~\eqref{eq:norm_gegen}, for which we use the Stirling's series as
in~\cite[Sec 3.3.1]{HaleTowsend}. The overall cost of the algorithm is
$O(n^2)$ operations. The accuracy of this algorithm, in turn, is
studied in section~\ref{error_estimates}.

\subsection{Multiple Intervals: An iterative Nystr\"om Method}\label{num_multi}
This section pre\-sents a spectrally accurate iterative Nystr\"om method
for the numerical solution of
equation~\eqref{eq:fraccionario_dirichlet} with $\Omega$ as
in~\eqref{union_intervals}.  This solver, which is based on use of the
equivalent second-kind Fredholm equation~\eqref{eq:Preconditioner},
requires (a)~Numerical approximation of $\mathcal{K}_s^{-1}f$,
(b)~Numerical evaluation of the ``forward-map''
$(I+\mathcal{K}_s^{-1}\mathcal{R}_s)[\tilde \phi]$ for each given
function $\tilde \phi$, and (c)~Use of the iterative linear-algebra
solver GMRES~\cite{GMRES}. Clearly, the algorithm in
Section~\ref{num_single} provides a numerical method for the
evaluation of each block in the block-diagonal inverse operator
$\mathcal{K}_s^{-1}$. Thus, in order to evaluate the aforementioned
forward map it now suffices to evaluate numerically the off-diagonal
operator $\mathcal{R}_s$ in equation~\eqref{other_intervals}.

An algorithm for evaluation of $\mathcal{R}_s[\tilde \phi](x)$ for $x\in
(a_j,b_j)$ can be constructed on the basis of the Gauss-Jacobi
quadrature rule for integration over the interval $(a_\ell,b_\ell)$
with $\ell\ne j$, in a manner entirely analogous to that described in
Section~\ref{num_single}. Thus, using Gauss-Jacobi nodes and weights
$y_i^{(\ell)}$ and $w_i^{(\ell)}$ ($i = 1,\dots, n_\ell$) for each
interval $(a_\ell,b_\ell)$ with $\ell\ne j$ we may construct a
discrete operator $R_n$ that can be used to approximate $\mathcal{R}_s[\tilde
\phi](x)$ for each given function $\tilde\phi$ and for all observation
points $x$ in the set of Gauss-Jacobi nodes used for integration in
the interval $(a_j,b_j)$ (or, in other words, for $x = y_k^{(j)}$ with
$k=1,\dots,n_j$).  Indeed, consideration of the numerical
approximation
$$ R[\tilde\phi](y_k^{(j)}) \approx \sum_{\ell \ne j} \sum_{i=0}^{n_\ell} |y_k^{(j)}-y_i^{(\ell)}|^{-2s-1} \tilde\phi(y_i^{(\ell)}) w_i^{(\ell)} $$
suggests the following definition. Using a suitable ordering to define
a vector $Y$ that contains all unknowns corresponding to
$\tilde\phi(y_i^{(\ell)})$, and, similarly, a vector $F$ that contains all
of the values $f(y_i^{(\ell)})$, the discrete equation to be solved
takes the form
\begin{equation*}
\label{eq:disc_precond}
(I + K_{s,n}^{-1} R_{s,n}) Y = K_{s,n}^{-1} [F]
\end{equation*}
where $R_n$ and $K_{s,n}^{-1}$ are the discrete operator that
incorporate the aforementioned ordering and quadrature rules.

With the forward map $(I + K_{s,n}^{-1} R_{s,n})$ in hand, the multi-interval
algorithm is completed by means of an application of a suitable
iterative linear algebra solver; our implementations are based on the
Krylov-subspace iterative solver GMRES~\cite{GMRES}.  Thus, the overall
cost of the algorithm is $O(m\cdot n^2)$ operations, where $m$ is the
number of required iterations. (Note that the use of an iterative
solver allows us to avoid the actual construction and inversion of the
matrices associated with the discrete operators in
equation~\eqref{eq:disc_precond}, which would lead to an overall cost of
the order of $O(n^3)$ operations.) As the equation to be solved
originates from a second kind equation, it is not unreasonable to
anticipate that, as we have observed without exception (and as
illustrated in Section~\ref{num_res}), a small number of GMRES
iterations suffices to meet a given error tolerance.

\subsection{Error estimates}\label{error_estimates}
The convergence rates of the algorithms proposed in
Sections~\ref{num_single} and~\ref{num_multi} are studied in what
follows. In particular, as shown in Theorems~\ref{teo_sobolev_error}
and~\ref{teo_analytic_error}, the algorithm's errors are exponentially
small for analytic $f$, they decay superalgebraically fast (faster
than any power of meshsize) for infinitely smooth right-hand sides,
and with a fixed algebraic order of accuracy $O(n^{-r})$ whenever $f$
belongs to the Sobolev space $H^r_s(\W)$ (cf. Section~\ref{Sobolev}).
For conciseness, fully-detailed proofs are presented in the
single-interval case only. A sketch of the proofs for the
multi-interval cases is presented in
Corollary~\ref{multi_interval_error}.

\begin{theorem}\label{teo_sobolev_error}
  Let $r > 0$, $0<s<1$. Then, there exists a constant $D$ such that
  the error $e_n(f) = (K_s^{-1} - K_{s,n}^{-1})(f)$ in the numerical
  approximation~\eqref{eq:geg_rec} for the solution of the single
  interval problem~\eqref{eqn_weighted} satisfies
\begin{equation}\label{singleinterval_sobolev_error}
  \| e_n(f) \|_{H^{\ell+2s}_s(a,b)} \le D n^{\ell-r} \| f \|_{H^r_s(a,b)}
\end{equation}  
for all $f \in H^r_s(a,b)$.  In particular, the $L^2_s$-bound
\begin{equation}\label{singleinterval_sobolev_L2}
 \| e_n(f) \|_{L^2_s(a,b)} \le D n^{-r} \| f \|_{H^r_s(a,b)}.
\end{equation} 
holds for every $f \in H^r_s(a,b)$.

\end{theorem}
\begin{proof} As before, we work with $(a,b)= (-1,1)$.
  Let $f$ be given and let $p_n$ denote the $n$-degree polynomial that
  interpolates $f$ at the Gauss-Gegenbauer nodes $(x_i)_{0\le i \le
    n}$. Since the Gauss-Gegenbauer quadrature is exact for
  polynomials of degree less or equal than $2n+1$, the approximate
  Gegenbauer coefficient $f_j^{(n)}$ (equation~\eqref{eq:disc_inner})
  coincides with the corresponding exact Gegenbauer coefficient of
  $p_n$: using the scalar product~\eqref{scalarproduct_L2} we have
  $$ f_j^{(n)} = \sum_{i=0}^n p_n(x_i) \tilde{C}_j^{(s+1/2)}(x_i)w_i = 
  \langle p_n, \tilde{C}_j^{(s+1/2)} \rangle_s. $$  It follows that
  the discrete operator $K_{s,n}$ satisfies $K_{s,n}^{-1}f =
  K_s^{-1}p_n$. Therefore, for each $\ell \ge 0$ we have
\begin{equation}\label{err_interpolatory}
  \| e_n(f) \|_{H^{\ell+2s}_s(-1,1)} = \| K_s^{-1}( f - p_n )\|_{H^{\ell+2s}_s(-1,1)} \le D_2 \| f - p_n \|_{H^{\ell}_s(-1,1)} ,
\end{equation}
where $D_2$ denotes the continuity modulus of the operator $K_s^{-1}:
H^{\ell}_s(-1,1) \to H^{\ell+2s}_s(-1,1)$ (see
Theorem~\ref{teo_extended} and
equation~\eqref{eq:weighted_fractional}).  From~\cite[Theorem
4.2]{GuoWang} together with the norm equivalence established in
Lemma~\ref{sobolev_equivalence}, we have, for all $\ell \le r$, the
following estimate for the interpolation error of a function $f \in
H^{r}_s(-1,1)$:
\begin{equation}\label{sobolev_interp_error}
\| f - p_n \|_{H_s^\ell(-1,1)} < C n^{\ell-r} \| f  \|_{H_s^r(-1,1)} \mbox{ for } f \in H_s^r(-1,1),
\end{equation}
which together with~\eqref{err_interpolatory} shows that~\eqref{singleinterval_sobolev_error} holds. The proof is complete.
\end{proof}
\begin{remark}\label{remark_negative_norm}
  A variety of numerical results in Section~\ref{num_res} suggest that
  the estimate~\eqref{singleinterval_sobolev_error} is of optimal
  order, and that the estimate~\eqref{singleinterval_sobolev_L2} is
  suboptimal by a factor that does not exceed $n^{-1/2}$.  In view of
  equation~\eqref{err_interpolatory}, devising optimal error estimates
  in the $L^2_s(a,b)$ norm is equivalent to that of finding optimal
  estimates for the interpolation error in the space $H_s^{-2s}(a,b)$.
  Such negative-norm estimates are well known in the context of
  Galerkin discretizations (see e.g.~\cite{BrennerScott}); the
  generalization of such results to the present context is left for
  future work.
\end{remark}

\begin{theorem}\label{teo_analytic_error}
  Let $f \in A_\rho$ be given (Definition~\ref{A_rho_def}) and let
  $e_n(f) = (K_s^{-1} - K_{s,n}^{-1})(f)$ denote the single-interval
  $n$-point error arising from the numerical method presented in
  Section~\ref{num_single}. Then the error estimate
\begin{equation}\label{singleinterval_analytic_error}
  \| e_n(f) \|_{A_\nu} \le C n^s \left(\frac{\nu}{\rho}\right)^n \| f \|_{A_\rho}, \quad \mbox{for all $\nu$ such that}\quad 1< \nu < \rho
\end{equation} 
holds. In particular, the operators $K_{s,n}^{-1}: A_\rho \to A_\rho$
converge in norm to the continuous operators $K_s^{-1}$ as
$n\to\infty$.
\end{theorem}
\begin{proof}
  Equations~\eqref{eq:weighted_fractional},
  \eqref{expansions},~\eqref{eq:disc_inner} and~\eqref{eq:geg_rec}
  tell us that
\begin{equation}\label{eq:separa_f}
(K_s^{-1} - K_{s,n}^{-1}) f = 
\sum_{j=0}^n \left(f_j-f_j^{(n)} \right) (\lambda_j^s)^{-1}  \tilde{C}_j^{(s+1/2)} +
\sum_{j=n+1}^\infty f_j (\lambda_j^s)^{-1}  \tilde{C}_j^{(s+1/2)}.
\end{equation}
In order to obtain a bound for the quantities $|f_j-f_j^{(n)}|$ we
utilize the estimate
\begin{equation}
\left| \int_{-1}^1 v(x) (1-x^2)^s dx - \sum_{i=0}^n v(x_i) w_i \right| \le 
\frac{C n^s}{\rho^{2n}} \|v\|_{L^\infty(\mathcal{B}_\rho)}.
\label{eq:quadrature}
\end{equation}
that is provided in~\cite[Theorem 3.2]{ZhaoWangXie} for the
Gauss-Gegenbauer quadrature error for a function $v \in
\mathcal{A}_\rho$. Letting $v = f \, \tilde{C}_j^{(s+1/2)}$ with $j\le n$,
equation~\eqref{eq:quadrature} and \eqref{eq:crecimiento_gegenbauer} yield
\begin{equation} 
\label{eq:quadrature_error_an}
 | f_j-f_j^{(n)}| \le
\frac{CD n^s}{\rho^{n}} \| f \|_{L^\infty(\mathcal{B}_\rho)}. 
\end{equation}
It follows that the infinity norm of the left-hand side in
equation~\eqref{eq:separa_f} satisfies
\begin{equation*}
\| (K_s^{-1} - K_{s,n}^{-1}) f \|_{L^\infty(B_\nu)} \leq C n^s \left( \frac{\nu}{\rho}\right)^{n} \|f\|_{L^\infty(\mathcal{B}_\rho)} \mbox{ for all } \nu < \rho
\label{eq:conv_analiticas}
\end{equation*}
for some (new) constant $C$, as it can be checked by
considering~\eqref{eq:crecimiento_gegenbauer},
\eqref{eq:quadrature_error_an} and Remark~\ref{eigenvalue_asymptotics}
for the finite sum in \eqref{eq:separa_f}, and~\eqref{eq:cota_an}
\eqref{eq:crecimiento_gegenbauer} and
Remark~\ref{eigenvalue_asymptotics} for the tail of the series. The
proof is now complete.

\end{proof}
\begin{corollary}\label{multi_interval_error}
  The algebraic order of convergence established in the
  single-interval Theorem~\ref{teo_sobolev_error} is valid in the
  multi-interval Sobolev case as well. Further, an exponentially small
  error in the infinity norm of $C^0(\Omega)$ results in the analytic
  multi-interval case (cf. Theorem~\ref{teo_analytic_error}).
\end{corollary}
\begin{proof}
  It is is easy to check that the family $\{R_{s,n}\}$ ($n\in\mathbb
  N$) of discrete approximations of the off-diagonal operator
  $\mathcal{R}_s$ is collectively compact~\cite{Kress} in the space
  $H^r_s(\Omega)$ ($r>0$). Indeed, it suffices to show that, for a
  given bounded sequence $\{\phi_{n}\}\subset H^r_s(\Omega)$, the
  sequence $R_{s,n}[\phi_{n}]$ admits a convergent subsequence in
  $H^r_s(\Omega)$. But, selecting $0<r' <r$, by
  Remark~\ref{sobolev_remark} we see that $\phi_{n}$ admits a
  convergent subsequence in $H^{r'}_s(\Omega)$. Thus, in view of the
  smoothness of the kernel of the operator $\mathcal{R}_s$, the bounds
  for the interpolation error~\eqref{sobolev_interp_error} applied to
  the product of $\phi_{n}$ and the kernel (and its derivatives), and
  the fact that the Gauss-Gegenbauer quadrature rule is exact for
  polynomials of degree $\leq 2n+1$, $R_{s,n}[\phi_{n}]$ converges in
  $H^t_s(\Omega)$ for all $t\in \mathbb{R}$ and, in particular for
  $t=r$. Thus, the family $\{R_{s,n}\}$ is collectively compact in
  $H^r_s(\Omega)$, as claimed, and therefore so is
  $K_{n,s}^{-1}R_{n,s}$. Then~\cite[Th. 10.12]{Kress} shows that, for
  some constant $C$, we have the bound
\begin{equation}\label{sobol_est}
\|\phi_n -\phi \|_{H^r_s}\leq C \| (\mathcal{K}_s^{-1}\mathcal{R}_s -
K_{n,s}^{-1}R_{n,s}) \phi \|_{H^r_s} + \| \mathcal{K}_s^{-1} -
K_{n,s}^{-1}) f \|.
\end{equation}

The proof in the Sobolev case now follows from~\eqref{sobol_est}
together with equations~\eqref{err_interpolatory}
and~\eqref{sobolev_interp_error} and the error estimates in
Theorem~\ref{teo_sobolev_error}. The proof in the analytic case,
finally, follows from the bound~\eqref{eq:quadrature},
Theorem~\ref{teo_analytic_error} and an application of
Theorem~\ref{gegen_embedding_Hr} to the left-hand side of
equation~\eqref{sobol_est}.
\end{proof}


\section{Numerical Results\label{num_res}}

This section presents a variety of numerical results that illustrate
the properties of algorithms introduced in Section~\ref{HONM}.  The
efficiency of these method is largely independent of the value of the
parameter $s$, and, thus, independent of the sharp boundary layers
that arise for small values of $s$.  To illustrate the efficiency of
the proposed Gegenbauer-based Nystr\"om numerical method and the
sharpness of the error estimates developed in Section~\ref{HONM}, test
cases containing both smooth and non-smooth right hand sides are
considered. In all cases the numerical errors were estimated by
comparison with reference solutions obtained for larger values of
$N$. Additionally, solutions obtained by the present Gegenbauer
approach were checked to agree with those provided by the
finite-element method introduced in~\cite{AcostaBorthagaray}, thereby
providing an independent verification of the correcteness of proposed
methodology.

\begin{center}
 \begin{figure}[ht]
\caption{Exponential convergence for $f(x) = \frac{1}{x^2 + 0.01}$.}
\includegraphics[scale=0.3]{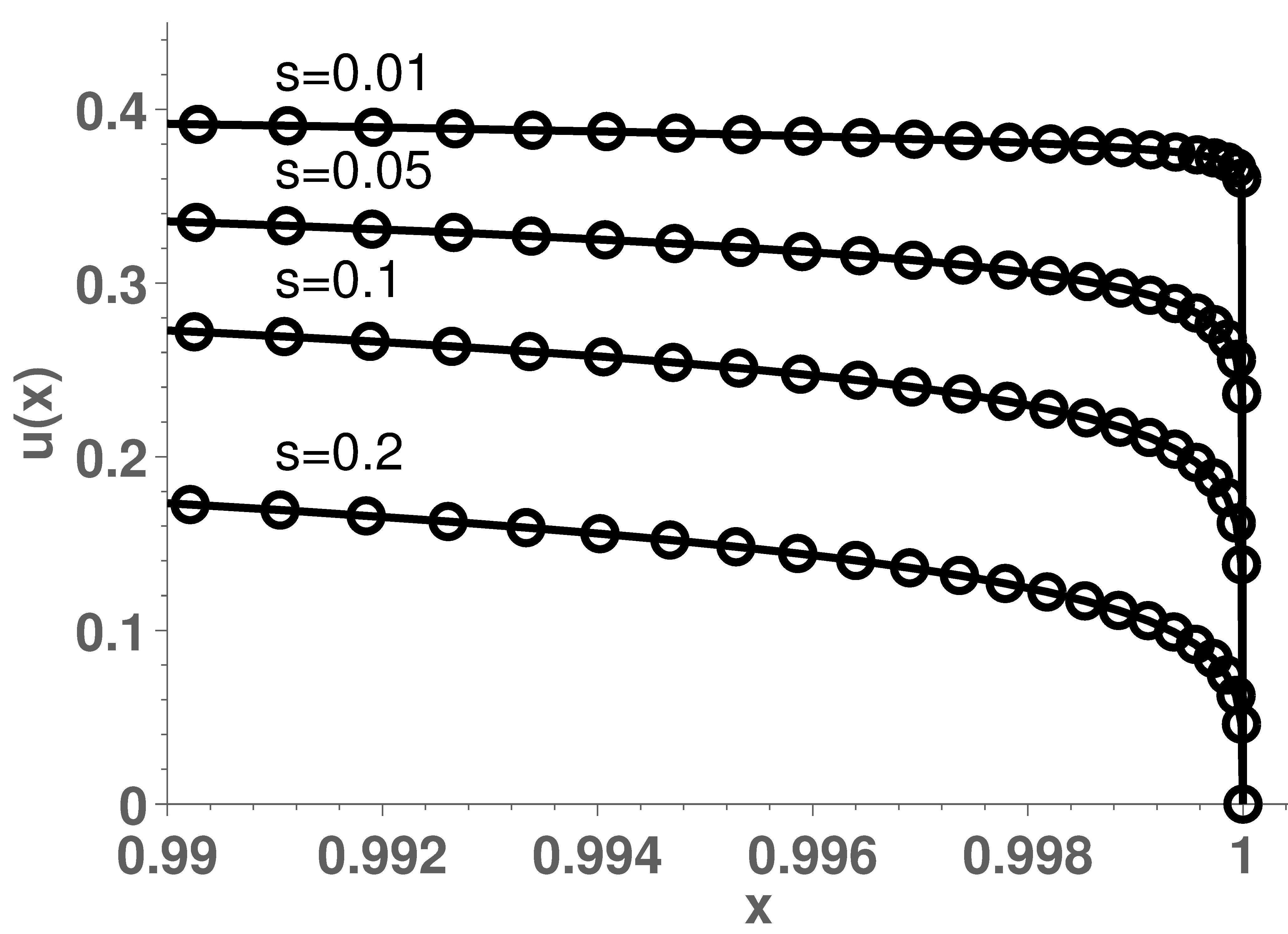}
\includegraphics[scale=0.3]{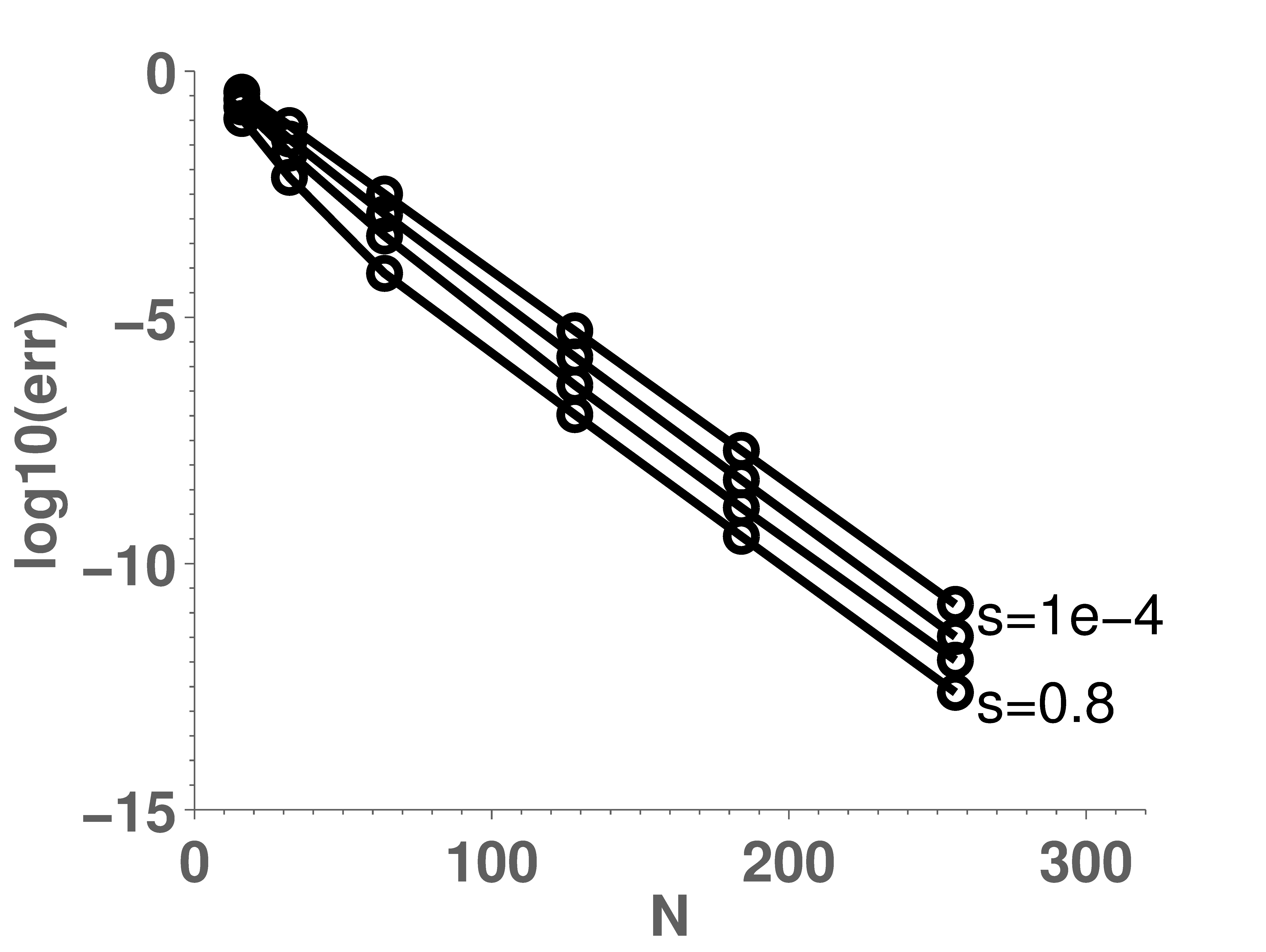} \\
\footnotesize Left: Solution detail near the domain boundary for $f$
equal to the Runge function mentioned in the text.  Right: Convergence
for various values of $s$.  Computation time: $0.0066$ sec. for $N=16$
to $0.05$ sec. for $N=256$.
\label{fig_runge}
\end{figure}
\end{center}

Figure~\ref{fig_runge} demonstrates the exponentially fast convergence
that takes place for a right-hand side given by the Runge function
$f(x) = \frac{1}{x^2 + 0.01}$---which is analytic within a small
region of the complex plane around the interval $[-1,1]$, and for
values of $s$ as small as $10^{-4}$.  The present Matlab
implementation of our algorithms produces these solutions with near
machine precision in computational times not exceeding $0.05$ seconds.
\begin{center}
 \begin{figure}[ht]
   \caption{Convergence in the $H^{2s}_s(-1,1)$ and $L_s^2(-1,1)$
     norms for $f(x) = |x|$. In this case, $f\in
     H^{3/2-\eps}_s(-1,1)$.}
\includegraphics[scale=0.35]{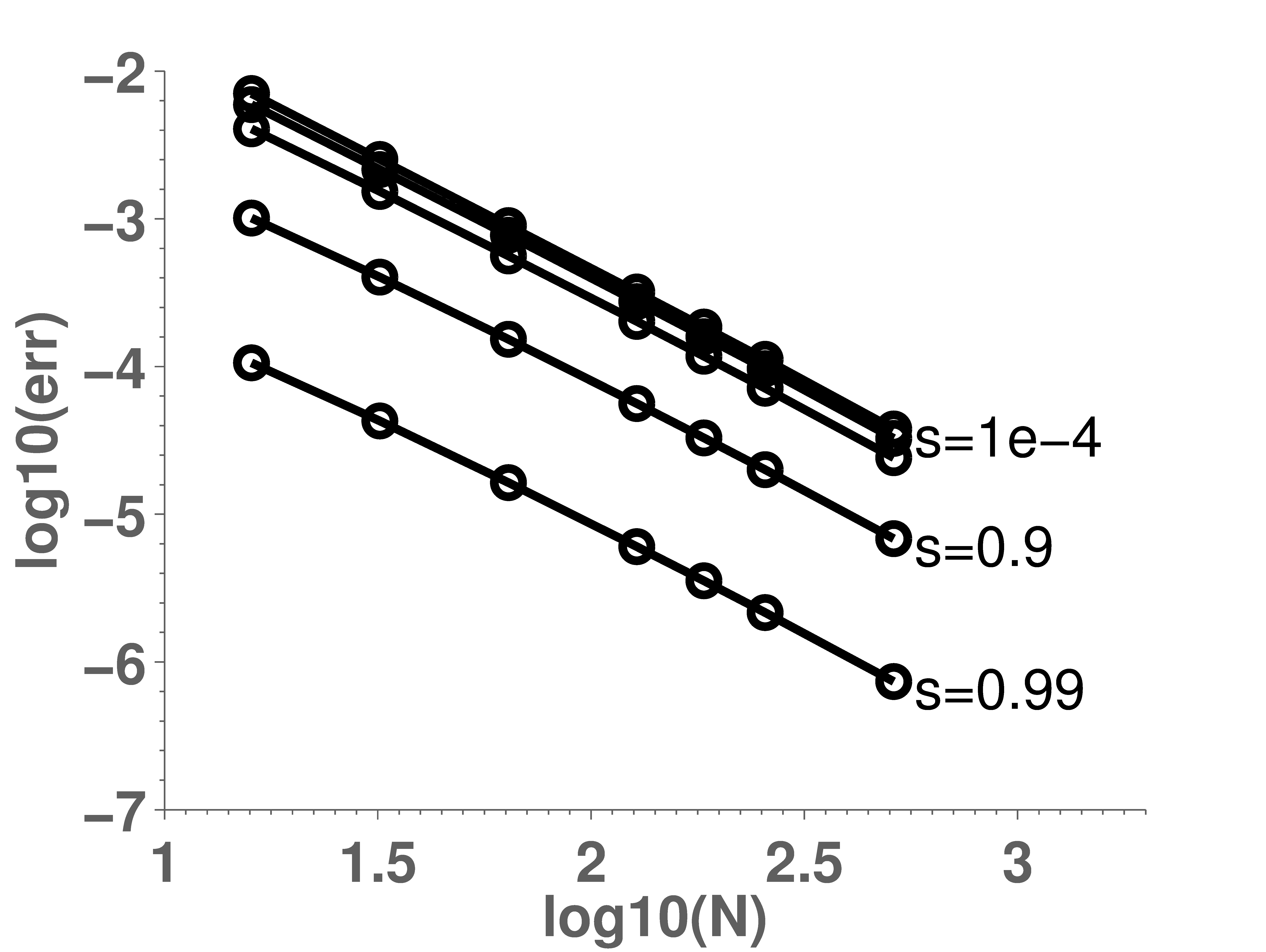}
\includegraphics[scale=0.35]{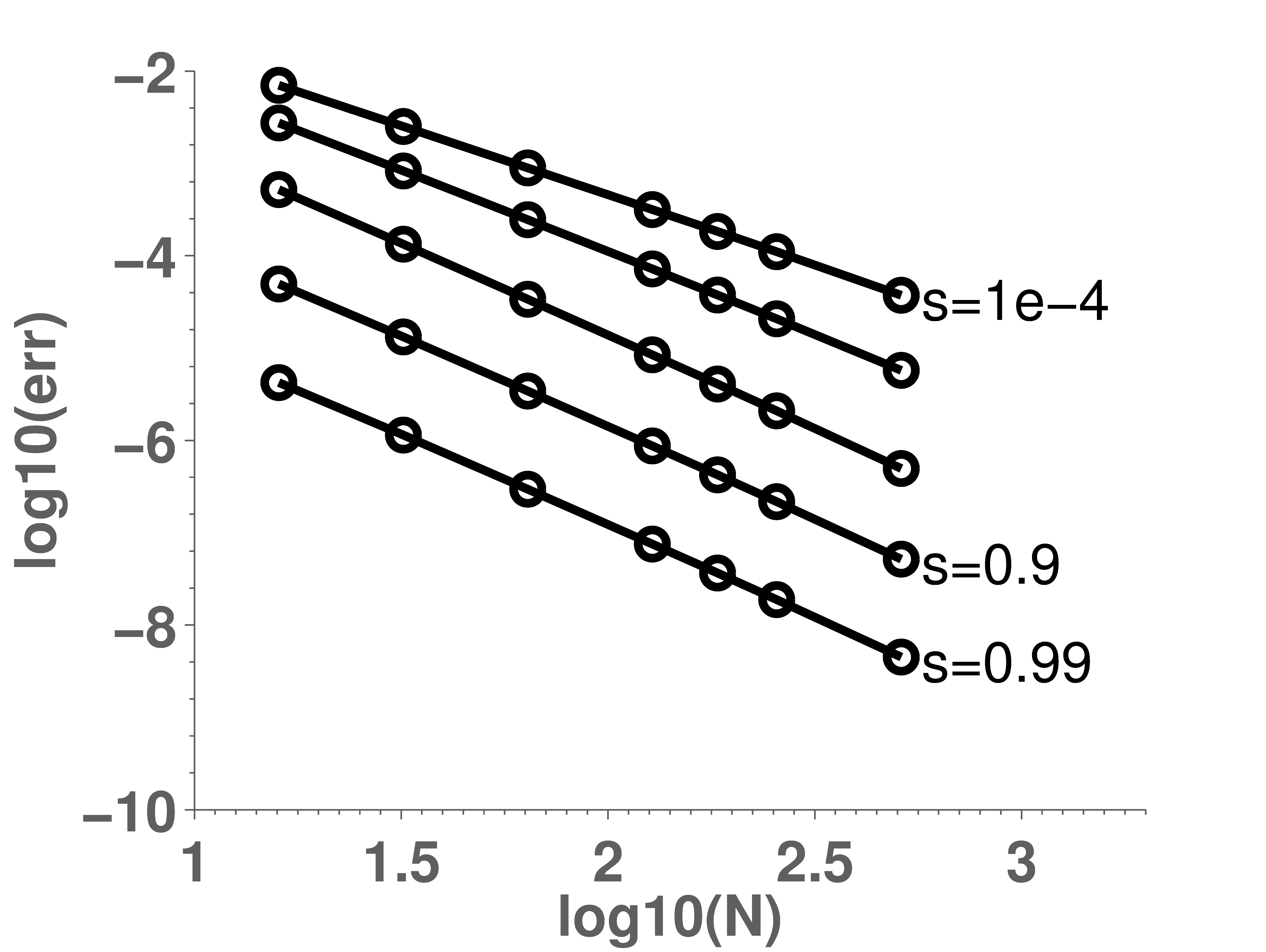}

\footnotesize
Left: errors in $H^{2s}_s(-1,1)$ norm of order $1.5$. Right: errors in $L^{2}_s(-1,1)$ norm, orders range from $1.5$ to $2$.
\label{fig_nonsmooth}
\end{figure}
\end{center}

Results concerning a problem containing the non-smooth right-hand side
$f(x) = |x|$ (for which, as can be checked in view of
Corollary~\ref{gegen_decay_coro} and Definition~\eqref{def:sobolev},
we have $f\in H^{3/2-\eps}_s(-1,1)$ for any $\eps >0$ and any $0\le s
\le 1$) are displayed in Fig.~\ref{fig_nonsmooth}. The errors decay
with the order predicted by Theorem~\ref{teo_sobolev_error} in the
$H^{2s}_s(-1,1)$ norm, and with a slightly better order than predicted
by that theorem for the $L_s^2(-1,1)$ error norm, although the
observed orders tend to the predicted order as $s\to 0$ (cf.
Remark~\ref{remark_negative_norm}).

%
 
\begin{center}
\begin{figure}[htbp]
  \begin{minipage}[t]{0.59\linewidth}
  \vspace{0pt}
  \centering
    \includegraphics[scale=0.27]{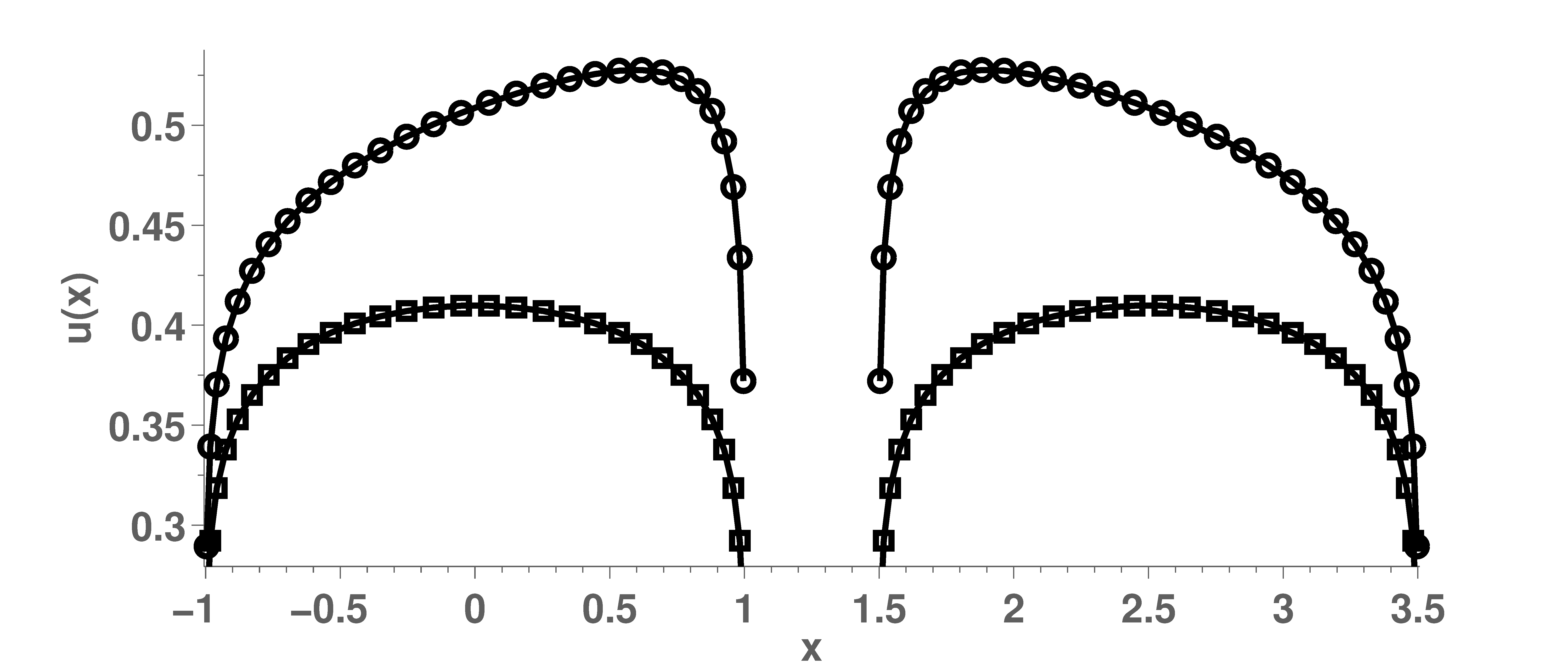}
  \end{minipage}
  \begin{minipage}[t]{0.4\linewidth}
  \vspace{0pt}
  \centering
      \begin{tabular}{ c | c }
      $N$ 	& rel. err. \\
      \hline
      \hline
	8      & 9.3134e-05    \\  
	12     & 1.6865e-06    \\  
	16     & 3.1795e-08    \\  
	20     & 6.1375e-10    \\  
	24     & 1.1857e-11    \\  
	28     & 1.4699e-13    \\
      \hline
      \end{tabular} 
\end{minipage}
\caption{Multiple (upper curves) vs. independent single-intervals
  solutions (lower curves) with $f=1$. A total of five GMRES
  iterations sufficed to achieve the errors shown on the right table
  for each one of the discretizations considered.}
\label{fig:multi}
\end{figure}
\end{center}
 
A solution for a multi-interval (two-interval) test problem with right
hand side $f=1$ is displayed in Figure~\ref{fig:multi}.  A total of
five GMRES iterations sufficed to reach the errors displayed for each
one of the discretizations considered on the right-hand table in
Figure~\ref{fig:multi}. The computational times required for each one
of the discretizations listed on the right-hand table are of the order
of a few hundredths of a second.

\section{Acknowledgments}

The authors thankfully acknowledge support from various agencies. GA
work was partially supported by CONICET, Argentina, under grant PIP
2014--2016 11220130100184CO. JPB's and MM's efforts were made possible
by graduate fellowship from CONICET, Argentina. OB efforts were
supported by the US NSF and AFOSR through contracts DMS-1411876 and
FA9550-15-1- 0043, and by the NSSEFF Vannevar Bush Fellowship under
contract number N00014-16-1-2808.

\appendix

\section{Appendix}
\subsection{Proof of Lemma \ref{lemma_exchangePV}} \label{app_exchangePV}Let $$F_\eps(x) = \int_{\Omega\setminus B_\eps(x)} \Phi_s(x-y) v(y) dy. $$ 
Then, by definition we have
\[
\lim_{\eps \to 0} \frac{d}{dx} F_\eps(x) =
P.V. \int_\Omega \frac{\partial}{\partial x} \Phi_s(x-y) v(y) dy.
\]
We note that interchanging the limit and differentiation processes on
the left hand side of this equation would result precisely in the
right-hand side of equation~\eqref{der_pv}---and the lemma would thus
follow. Since $F_\eps$ converges throughout $\Omega$ as
$\eps\to 0$, to show that the order of the limit and
differentiation can indeed be exchanged it suffices to
show~\cite[Th. 7.17]{baby_rudin} that the quantity $\frac{d}{dx}
F_\eps(x)$ converges uniformly over compact subsets $K\subset
\Omega$ as $\eps\to 0$.

To establish the required uniform convergence property over a given
compact set $K\subset \Omega$ let us first define a larger compact set
$K^*=[a,b]\subset \Omega$ such that $K\subset U \subset K^*$ where $U$
is an open set.  Letting $\eps_0$ be sufficiently small so that
$B_{\eps_0}(x) \subset K^*$ for all $x\in K$, for each
$\eps < \eps_0$ we may then write
\begin{equation*}
\label{eq_split_kstar}
\frac{\partial}{\partial x} F_\eps  =  \int_{\Omega\setminus K^*} \frac{\partial}{\partial x} \Phi_s(x-y) v(y)dy +  
\int_{K^* \setminus B_\eps(x)} \frac{\partial}{\partial x} \Phi_s(x-y) v(y)dy.
\end{equation*}
The first term on the right-hand side of this equation does not depend
on $\eps$ for all $x\in K$.  To analyze the second term we
consider the expansion $v(y) = v(x) + (y-x)R(x,y)$ and we write
$\int_{K^*} \frac{\partial}{\partial x} \Phi_s(x-y) v(y)dy =
\Gamma^1_\eps(x) + \Gamma^2_\eps(x)$ where
\begin{equation*} \begin{split}
\Gamma^1_\eps(x) & = v(x)\int_{K^* \setminus B_\eps(x)} \frac{\partial}{\partial x} \Phi_s(x-y) dy \quad \mbox{and} \\
\Gamma^2_\eps(x) & =  \int_{K^* \setminus B_\eps(x)} \frac{\partial}{\partial x} \Phi_s(x-y) (y-x)R(x,y)dy.
\end{split}
\end{equation*}
Since $K^*=[a,b]$, for each $\eps < \eps_0$ and each
$x\in K$ the quantity $\Gamma^1_\eps(x)$ can be expressed in the form
$$ \Gamma^1_\eps(x) = -v(x) \left( \Phi_s(x-y)\big|^{y=b}_{y =
    x+\eps} + \Phi_s(x-y)\big|^{y =x-\eps}_{y =a}
\, \right) $$
which, in view of the relation $\Phi_s(-\eps) = \Phi_s(\eps)$,
is independent of $\eps$. The uniform convergence of
$\Gamma^1_\eps(x)$ over $K$ therefore holds trivially.

The term $\Gamma^2_\eps(x)$, finally, equals
$$\int_{K^* \setminus B_\eps(x)}N(x-y)R(x,y),$$ 
where $N(x,y) = \frac{\partial}{\partial x} \Phi_s(x-y)(y-x)$. Since
$v\in C^1(\Omega)$ there exists a constant $C_{K,K^*}$ such that
$|R(x,y)|<C_{K,K^*}$ for all $(x,y)$ in the compact set $K \times K^*
\subset \Omega \times \Omega$. In particular, for each $x\in K$ the
product $N(x-y)R(x,y)$ is integrable over $K^*$, and therefore the
difference between $\Gamma^2_\eps(x)$ and its limit satisfies
$$\left|\Gamma^2_\eps(x) - \lim_{\eps \to 0} \Gamma^2_\eps(x) \right| = 
\left|\int_{x-\eps}^{x+\eps} N(x-y) R(x,y) dy\right| < C_{K,K^*}
\int_{-\eps}^{\eps} \left| N(z) \right|dz.$$
The uniform convergence of $\Gamma^2_\eps$ over the set $K$
then follows from the integrability of the function $N=N(z)$ around
the origin, and the proof is thus complete.

\subsection{Interchange of infinite summation and P.V. integration in
  equation
  \eqref{eq:series_Ks}} \label{appendix_pvseries}
\begin{lemma}
  Upon substitution of~\eqref{taylor_one_bnd}, the quantity $L^s_n$ in
  equation~\eqref{eq:Ks_n} equals the expression on the right-hand
  side of equation~\eqref{eq:series_Ks}. In detail, for each
  $x\in(0,1)$ we have
\begin{equation}\label{app_form}
  P.V. \int_{0}^{1} J^s(x-y) y^{s-1}  \left(\sum_{j=0}^{\infty}q_j y^j \right) y^n dy =  \sum_{j=0}^{\infty}  \left( \mbox{P.V.} \int_0^1 q_j y^j J^s(x-y)y^{s-1}y^n dy \right),
\end{equation}
where $J^s(z) = \sgn(z)|z|^{-2s}$.
\end{lemma}
\begin{proof}
  Let $x\in (0,1)$ be given. Then, taking $\delta < \min \{ x,1-x\}$
  we re-express the left hand side of~\eqref{app_form} in the form
\begin{equation}\label{app_form_2} \lim_{\eps \to 0} \left[\int_0^\delta dy + \int_{[\delta,1-\delta]\setminus B_\eps(x)} dy    + \int_{1-\delta}^1 dy \right] \left(\sum_{j=0}^{\infty} J^s(x-y)  q_j y^{s-1+n+j} \right ). 
\end{equation}
The leftmost and rightmost integrals in this expression are
independent of $\eps$, and, in view of \eqref{taylor_one_bnd}, they
are both finite.  The exchange of these integrals and the
corresponding infinite sums follows easily in view of the monotone
convergence theorem since the coefficients $q_j$ are all positive.


The middle integral in equation~\eqref{app_form_2}, in turn, can be
expressed in the form
\begin{equation}
\label{innter_integrals}
\lim_{\eps \to 0} \int_{[\delta,1-\delta]\setminus B_\eps(x)} J^s(x-y) \left( \lim_{m\to \infty} v_m(y) \right) dy,
\end{equation}
where 
\begin{equation}\label{vm_sum}
  v_m(y) = y^{s-1} y^n \sum_{j=0}^m q_j y^j.
\end{equation} In view
of~\eqref{taylor_one_bnd}, $v_m$ converges (uniformly) to the smooth
function $v_\infty(y) = y^{s-1} y^n(1-y)^{s-1}$ for all $y$ in the
present domain of integration.  As shown below, interchange of this uniformly
convergent series with the PV integral will then allow us to  
complete the proof of
the lemma. 

In order to justify this interchange we replace the expansion
$$ v_m(y) = v_m(x) + (x-y)R_m(x,y),  \quad \mbox{where } R_m(x,y) = \frac{v_m(y) - v_m(x)}{x-y}. $$
in~\eqref{innter_integrals} and we define
\begin{align}
\label{pv_vmx}
&F^1_\eps = v_\infty(x) \int_{[\delta,1-\delta]\setminus B_\eps(x)}
J^s(x-y) dy 
\\
\label{int_xyRm}
&F^2_\eps = \int_{[\delta,1-\delta]\setminus B_\eps(x)} J^s(x-y)(x-y)
\lim_{m\to\infty} R_m(x,y)dy;
\end{align}
clearly the expression in equation~\eqref{innter_integrals} equals
$\lim_{\eps \to 0} \left( F^1_\eps + F^2_\eps\right)$.

The exchange of $\lim_{\eps \to 0}$ and infinite summation for
$F^1_\eps$ (in~\eqref{pv_vmx}) follows immediately since $v_m(x)$
does not depend on $\eps$. In order to perform a similar exchange
for $F^2_\eps$ in~\eqref{int_xyRm} we first note that
\begin{equation}\label{lim_f2}
  \lim_{\eps \to 0} F^2_\eps =  \int_\delta^{1-\delta} J^s(x-y)(x-y) \lim_{m\to\infty} R_m(x,y)dy 
\end{equation} 
in view of the integrand's integrability---which itself follows from
the bound
\begin{equation}\label{bound}
  \left | J^s(x-y)(x-y) \lim_{m\to\infty} R_m(x,y)\right |\leq M  \left | J^s(x-y)(x-y)\right |,
\end{equation}
(where $M$ is a bound for the derivative $\left[ v_\infty(y)\right]'$
in the interval $[\delta,1-\delta]$) together with the integrability
of the product $\left | J^s(x-y)(x-y)\right |$.  But~\eqref{lim_f2}
equals
\begin{equation}\label{reverse_lim}\begin{split}
  \lim_{m\to\infty}\int_\delta^{1-\delta}& J^s(x-y)(x-y) R_m(x,y)dy = \\ 
	& = \lim_{m\to\infty}\lim_{\eps \to 0}\int_{[\delta,1-\delta]\setminus B_\eps(x)}J^s(x-y)(x-y)  R_m(x,y)dy.
\end{split} \end{equation}
Indeed, the first expression results from an application of the
dominated convergence theorem---which is justified in view
of~\eqref{bound} since $R_m(x,y)$ is an increasing sequence---while
the second equality, which puts our integral in ``principal value''
form, follows directly in view of the integrand's integrability.

The lemma now follows by substituting first $R_m(x,y) = (v_m(y) -
v_m(y))/(x-y)$ and then equation~\eqref{vm_sum} in the right-hand
integral of equation~\eqref{reverse_lim} and combining the result with
corresponding sums for $F^1_\eps$ and for the leftmost and rightmost
integrals in~\eqref{app_form_2}---to produce the desired right-hand
side in equation~\eqref{app_form}. The proof is now complete.
\end{proof}
\subsection{Interchange of summation order in
  \eqref{eq:series_ajk} for $x\in
  (0,1)$} \label{appendix_sumorder} 

Letting
$$a_{jk} = \frac{(1-s)_j}{j!} \frac{(2s)_k}{k!} \frac{1}{s-n-j+k} \, x^k, $$
in order to show that the summation signs in~\eqref{eq:series_ajk} can
be interchanged it suffices to show that the series $\sum_{j,k}
a_{jk}$ is absolutely convergent. To do this we write
\begin{align*}
\sum_{j=0}^\infty & |a_{jk}| =  \frac{(2s)_k}{k!} \, x^k \sum_{j=0}^\infty \frac{(1-s)_j}{j!} \frac{1}{|s-n-j+k|} = \\
& = \frac{(2s)_k}{k!} x^k \left( \sum_{j=0}^{ k-n }\frac{(1-s)_j}{j!} \frac{1}{s-n-j+k} 
+ \sum_{j= k-n+1 }^\infty \frac{(1-s)_j}{j!} \frac{1}{-s+n+j-k} \right). 
\end{align*}

Since $\frac{(1-s)_j}{j!}\sim j^{-s}$ as $j \to \infty$  
we obtain
$$\sum_{j = k-n+1 }^\infty \frac{(1-s)_j}{j!} \frac{1}{-s+n+j-k} \leq C \sum_{j = k-n+1 }^\infty \frac{j^{-s}}{-s+n+j-k}  \leq C(s)$$
and, in view of the fact that, in particular, $\frac{(1-s)_j}{j!}$ is
bounded,
$$\sum_{j=0}^{ k-n }\frac{(1-s)_j}{j!} \frac{1}{s-n-j+k} \leq \sum_{\ell=0}^{ k-n } \frac{1}{s+\ell} = \frac{1}{s} + \sum_{\ell=1}^{ k-n } \frac{1}{s+\ell}.$$
It follows that
$$\sum_{k=0}^\infty\sum_{j=0}^\infty |a_{jk}| \leq \sum_{k=0}^\infty \frac{(2s)_k}{k!} \left(C(s) + \sum_{\ell=1}^{k-n} \frac{1}{\ell} \right) x^k $$
and, since $\frac{(2s)_k}{k!} \sim k^{2s-1}$ and $\sum_{\ell=1}^{k-n}
\frac{1}{\ell} \sim \ln k $ as $k \to \infty$, the sum $\sum_{j,k}
a_{jk}$ is absolutely convergent for every $x \in (0,1)$, as needed.


\bibliographystyle{plain}
\bibliography{bib}

\end{document}